\documentclass[12pt, a4paper]{amsart}
\usepackage{amsmath}
\usepackage{geometry,amsthm,graphics,tabularx,amssymb,shapepar}
\usepackage{amscd}
\usepackage{graphicx}
\usepackage[all]{xypic}
\usepackage{hyperref}
\usepackage{color}
\usepackage{soul}
\usepackage{bm}

\newcommand{\fk}{\mathfrak k}
\newcommand{\fg}{\mathfrak g}

\newcommand{\dg}{\dot{\mathfrak g}}
\newcommand{\dfh}{\dot{\mathfrak h}}
\newcommand{\dQ}{\dot{Q}}
\newcommand{\dE}{\dot{E}}
\newcommand{\al}{\alpha}

\newcommand{\wh}{\widehat}
\newcommand{\uce}{\mathfrak{uce}}
\newcommand{\Aut}{\mathrm{Aut}}
\newcommand{\dmu}{\dot{\mu}}

\newcommand{\dal}{\dot{\alpha}}

\newcommand{\ot}{\otimes}
\newcommand{\CK}{\mathcal{K}}

\newcommand{\CL}{\mathcal{L}}

\newcommand{\C}{\mathbb{C}}
\newcommand{\R}{\mathbb R}

\newcommand{\Z}{\mathbb Z}
\newcommand{\rd}{\mathrm{d}}
\newcommand{\rk}{\mathrm{k}}

\newcommand{\la}{\langle}
\newcommand{\ra}{\rangle}

\newcommand{\ba}{\begin {eqnarray}}
\newcommand{\ea}{\end {eqnarray}}
\newcommand{\baa}{\begin {eqnarray*}}
\newcommand{\eaa}{\end {eqnarray*}}
\newcommand{\be}{\begin {equation}}
\newcommand{\ee}{\end {equation}}
\newcommand{\bee}{\begin {equation*}}
\newcommand{\eee}{\end {equation*}}

\newcommand{\te}[1]{\textnormal{{#1}}}

\theoremstyle{Theorem}

\theoremstyle{Theorem}

\newtheorem{thm}{Theorem}[section]
\newtheorem{cort}[thm]{Corollary}
\newtheorem{lemt}[thm]{Lemma}
\newtheorem{prpt}[thm]{Proposition}

\theoremstyle{Theorem}

\theoremstyle{Theorem}

\theoremstyle{Plain}

\theoremstyle{Definition}

\newtheorem{dfnt}[thm]{Definition}

\def\({\left(}

\def\){\right)}

\newlength{\dhatheight}

\newcommand{\set}[2]{{
    \left.\left\{
        {#1}
    \,\right|\,
        {#2}
    \right\}
}}

\def \<{{\langle}}
\def \>{{\rangle}}

\allowdisplaybreaks

\numberwithin{equation}{section}
\title[MRY presentations]{Twisted toroidal Lie algebras and Moody-Rao-Yokonuma presentation}

\author{Fulin Chen$^1$}
\address{School of Mathematical Sciences, Xiamen University,
 Xiamen, China 361005} \email{chenf@xmu.edu.cn }
\thanks{$^1$Partially supported by NSF of China (No.11501478).}

\author{Naihuan Jing$^2$}
\address{Department of Mathematics, North Carolina State University, Raleigh, NC 27695,
USA}
\email{jing@math.ncsu.edu}
\thanks{$^2$Partially supported by NSF of China (No.11531004) and Simons Foundation (No.198129).}

\author{Fei Kong$^3$}
\address{College of Mathematics and Statistics, Hunan Normal University, Changsha, China 410006} \email{kfkfkfc@outlook.com}
\thanks{$^3$Partially supported by NSF of China (No.11701183).}

 \author{Shaobin Tan$^4$}
 \address{School of Mathematical Sciences, Xiamen University,
 Xiamen, China 361005} \email{tans@xmu.edu.cn}
 \thanks{$^4$Partially supported by NSF of China (No.11471268, No.11531004).}


\begin{document}

\begin{abstract} Let $\fg$ be an affine Kac-Moody algebra,
and $\mu$ a diagram automorphism of $\fg$.
 In this paper,  we give an explicit realization for the universal central extension $\wh\fg[\mu]$ of the twisted loop algebra of $\fg$ related to $\mu$, which
 provides a Moody-Rao-Yokonuma presentation for the algebra $\wh\fg[\mu]$
 when $\mu$ is non-transitive, and the presentation is indeed related to the quantization of toroidal Lie algebras.
\end{abstract}
\maketitle

\section{Introduction}

Let $\fg$ be a (twisted or untwisted) affine Kac-Moody algebra (without derivation),
and $\bar\fg$ the quotient algebra of $\fg$ modulo its center.
When $\fg$ is of untwisted type, the universal central extension $\wh\fg$
of the loop algebra $\C[t_1,t_1^{-1}]\ot\bar\fg$ is called a
{\it toroidal Lie algebra}.
This algebra was first introduced by Moody-Rao-Yokonuma in \cite{MRY}, where the authors introduced the famous
Moody-Rao-Yokonuma (MRY) presentation.  
The presentation makes it more effective to 
study representations of toroidal Lie algebras in a manner similar
to that of untwisted affine Lie algebras \cite{MRY, Tan1, Tan2, FJW, JM-fermionic-real-toro,JMX, C,JMT}.
Moreover, it turns out that the classical limit of the quantum toroidal algebra
is just the MRY presentation of the toroidal Lie algebra (\cite{GKV,J-KM}).

Let $\mu$ be a diagram automorphism of $\fg$ of order $N$,
and $\bar\mu$  the automorphism on $\bar\fg$ induced from $\mu$. The twisted loop algebra $\CL(\bar\fg,\bar\mu)$ of $\bar\fg $ is defined as follows
\begin{align*}
  \CL(\bar\fg,\bar\mu)=\bigoplus_{n\in\Z}\C t_1^n\ot\bar\fg_{(n)},
\end{align*}
where $\bar\fg_{(n)}=\set{x\in\bar\fg}{\bar\mu(x)=\xi^nx}$ and $\xi=e^{2\pi\sqrt{-1}/N}$.
In this paper, we study the
 universal central
extension $\wh\fg[\mu]$  of $\CL(\bar\fg,\bar\mu)$, and give Moody-Rao-Yokonuma presentation for $\wh\fg[\mu]$ when $\mu$ is non-transitive. One may expect that the MRY presentation could be used to study the representation and quantization for the twisted toroidal extended affine Lie algebras (\cite{MRY, GKV,J-KM, BL}).

An \emph{extended affine Lie algebra} (EALA)
 is a complex Lie algebra,
together with a non-zero finite dimensional Cartan subalgebra and a non-degenerate invariant symmetric bilinear form, that satisfies a list of natural axioms (\cite{H-KT,AABGP, N2}). The root system of an EALA is a disjoint union of isotropic and non-isotropic root systems, and the rank of the free abelian group generated by the isotropic root system is defined to be the {\it nullity} of the EALA (\cite{AABGP}).
 Indeed, the nullity 0 EALAs are finite dimensional simple Lie algebras over the complex number field, and the
nullity $1$ EALAs are precisely the affine Kac-Moody
algebras (\cite{ABGP}). We remark that the  nullity 2 EALAs are the most important class of EALAs other than the finite dimensional simple Lie algebras and affine Kac-Moody algebras, which are closely related to the singularity theory studied by Saito and Slodowy (\cite{Saito-EALA, Slo}). And the nullity 2 EALAs are classified in \cite{ABP} (also see \cite{GP-torsors}).

For a given EALA $\mathfrak L$, the subalgebra of $\mathfrak L$ generated by the set of non-isotropic root vectors is called the {\it core} of $\mathfrak L$ (\cite{AABGP}).
We denote by $\mathbb{E}_2$ the class of all Lie algebras that are isomorphic
to the {\it centerless cores} (cores modulo their centers) of EALAs with nullity $2$.
Let $\mathfrak{sl}_n(\C_q)$ ($n\ge 2$) be the special linear Lie algebra over
the quantum torus $\C_q$ in two variables (\cite{BGK}).
It was proved in \cite{ABP} that
any Lie algebra in $\mathbb{E}_2$ is either isomorphic to $\mathfrak{sl}_n(\C_q)$ with $q\in \C^{\times}$ not a root of unity,
or isomorphic to a Lie algebra of the form $\CL(\bar\fg,\bar\mu)$ with $\mu$ non-transitive. The universal central extension $\wh{\mathfrak{sl}}_n(\C_q)$ of $\mathfrak{sl}_n(\C_q)$ is given in \cite{BGK}, and its MRY presentation is obtained in \cite{VV-double-loop}. The purpose of this paper is to study the universal central extension $\wh\fg[\mu]$ of $\CL(\bar\fg,\bar\mu)$, and the MRY presentation for $\wh\fg[\mu]$ with  non-transitive diagram automorphism $\mu$.

This paper is organized as follows.
In Section 2, we recall some facts for the affine Kac-Moody algebras which will be used later on.
In Section 3, we show that any diagram automorphism $\mu$ of an affine Kac-Moody algebra $\fg$ can be lifted to an automorphism
$\wh\mu$ for the universal central extension $\wh\fg$ of $\CL(\fg,{\rm{id}})$.
The Lie subalgebra of $\wh\fg$ fixed by $\wh\mu$ is denoted by $\wh\fg[\mu]$.
We claim that $\wh\fg[\mu]$ is the universal central extension of $\CL(\bar\fg,\bar\mu)$
(Theorem \ref{main1}), and state the MRY presentation for $\wh\fg[\mu]$
with $\mu$ non-transitive (Theorem \ref{main2}).
Section 4 and Section 5 are devoted to the proofs of Theorems \ref{main1}
and \ref{main2}.

We denote the sets of non-zero complex numbers, non-zero integers, and positive integers respectively by
 $\C^\times$, $\Z^\times$ and $\Z_+$. For  $M\in \Z_+$, we set $\xi_M=e^{2\pi\sqrt{-1}/M}$ and $\Z_M=\Z/M\Z$.

\section{Diagram automorphisms of affine Kac-Moody algebras}

\subsection{Affine Kac-Moody algebras}

In this subsection, we collect some basics about affine Kac-Moody algebras
that will be used later on.

Let $A=(a_{ij})_{i,j=0}^\ell$ be a {\it generalized Cartan matrix} (GCM) of affine type, and $\fg$ the affine Kac-Moody algebra (without derivation) associated to the GCM $A$.
We denote the set $\{0,1,\dots,\ell\}$ by $I$.
By definition, the Lie algebra $\fg$ is generated by the Chevalley generators
$\al_i^\vee, e_i^\pm, i\in I$ with the defining relations $(i,j\in I)$
\begin{align*}
[\al_i^\vee,\al_j^\vee]=0,\ [\al_i^\vee, e_j^\pm]=\pm a_{ij}\, e_j^\pm,\
[e_i^+,e_j^-]=\delta_{ij}\al_i^\vee,\ \mathrm{ad}(e_i^\pm)^{1-a_{ij}}(e_j^\pm)=0\ (i\ne j).
\end{align*}

%
%

Let $\Delta$ be the root system (including $0$) of $\fg$, $\Delta^{\times}$
 the set of
real  roots in $\Delta$, and $\Delta^0=\Delta\setminus \Delta^\times=\Z\delta_2$ the set of imaginary roots in $\Delta$.
Then $\fg$ has a root space decomposition $\fg=\oplus_{\al\Delta}\fg_\al$.
Let $\Pi=\{\al_i, i\in I\}$ be the simple root system of $\fg$ such that $e_i^\pm\in \fg_{\pm\al_i}$ for $i\in I$, and
$Q=\oplus_{i\in I}\Z\al_i$ the root lattice of $\fg$.
Then the root space decomposition naturally induces a $Q$-grading on $\fg$.
In addition, let $\bar\fg$ be the quotient algebra of $\fg$ modulo its center.
Then the $Q$-grading on $\fg$ naturally
induces a $Q$-grading $\bar\fg=\oplus_{\al\in Q}\bar\fg_\al$ on $\bar\fg$.

Now we recall the twisted loop realization of the affine Kac-Moody algebra $\fg$ (see \cite[Chapters 7 and 8]{Kac-book}).
Using the notations given in \cite[Chapter 4. Table Aff 1-3]{Kac-book}, we assume that the GCM $A$ is of type $X_n^{(r)}$.

We start with a finite dimensional simple Lie algebra $\dg$ of type $X_n$.
Let $\dot\al_i^\vee,\dot E_i^\pm$ $(i=1,2,\dots,n)$ be the Chevalley generators
of $\dg$,
and $\dfh=\oplus_{i=1}^n\C\dot\al_i^\vee$ a Cartan subalgebra of $\dg$.
We denote by $\dot\Delta$ the root system (containing $0$) of $\dg$ with respect to $\dfh$.
Then $\dg$ has a root space decomposition $\dg=\oplus_{\dot\al\in\dot\Delta}\dg_{\dot\al}$ such that $\dg_0=\dfh$.
Let $\dot\Pi$ be a fixed simple root system of $\dot\Delta$,
and $\dot\Delta_+$ the set of positive roots with respect to $\dot\Pi$.
In addition, for each $\dot\al\in \dot\Delta_+$, there exist $\dot E_{\dot\al}^\pm\in\dg_{\pm\dot\al}$ and $\dot\al^\vee\in\dfh$, such that
$\{\dot E_{\dot\al}^+,\dot\al^\vee,\dot E_{\dot\al}^-\}$
form a $\mathfrak{sl}_2$ triple.
Moreover, for a simple root $\dot\al_i\in\dot \Pi$, we assume that
$\dot E_{\dot\al_i}^\pm=\dot E_i^\pm$.

Let $\dot\nu$ be a diagram automorphism of $\dg$ of order $r$.
By definition, there exists a permutation $\dot\nu$ on the set $\{1,2,\dots,n\}$, such that
\begin{align*}
  \dot\nu(\dot E_i^\pm)=\dot E_{\dot\nu(i)}^\pm\quad\te{and}\quad
  \dot\nu(\dot\al_i^\vee)=\dot\al_{\dot\nu(i)}^\vee\quad\te{for }
  i=1,2,\dots,n.
\end{align*}
For each $x\in\dg$ and $m\in\Z$, we set
\begin{align*}
x_{[m]}=r^{-1}\sum_{p\in \Z_r}\xi_r^{-mp}\dot{\nu}^p(x)\quad\te{and}\quad
\dg_{[m]}=\{x_{[m]}\mid x\in \dg\}.
\end{align*}
And define the Lie algebra
\begin{align*}
\te{Aff}(\dg,\dot\nu)=\bigoplus_{m\in\Z}\C t_2^m\ot\dg_{[m]}\oplus\C\rk_2
\end{align*}
with the Lie bracket given by
\begin{align*}
[t_2^{m_1}\ot x+a_1\rk_2, t_2^{m_2}\ot y+a_2\rk_2]
=t_2^{m_1+m_2}\ot [x,y]+\<x,y\>\delta_{m_1+m_2,0}m_1\rk_2,
\end{align*}
where $m_1,m_2\in \Z$, $x\in \dg_{[m_1]},y\in \dg_{[m_2]}$, $a_1,a_2\in \C$ and
$\<\cdot,\cdot\>$ is the normalized symmetric invariant bilinear form on $\dg$.

We denote by
  \begin{equation*}\dot\theta=\begin{cases}\te{the highest root of}\ \dg,\ &\te{if}\ r=1\ \te{or}\ X_n=A_{2\ell}, r=2;\\
  \dot{\al}_1+\cdots+\dot{\al}_\ell,\ &\te{if}\ X_n=D_{\ell+1}, r=2,3;\\
  \dot{\al}_1+\cdots+\dot{\al}_{2\ell-2}, &\te{if}\ X_n=A_{2\ell-1}, r=2;\\
  \dot{\al}_1+2\dot{\al}_2+2\dot{\al}_3+\dot{\al}_4+\dot{\al}_5+\dot{\al}_6,&\te{if}\ X_n=E_6, r=2.
   \end{cases}\end{equation*}
And for each $i=1,2,\dots,n$, we let $r_i$ be
the cardinality of the set $\{\dot\nu^k(i)\,|\,k\in\Z_r\}$.
If the GCM $A$ is of type $A_{2\ell}^{(2)}$, we set
\begin{align*}E_i^\pm=r_i \dE_{i[0]}^\pm,\  E_\ell^\pm=\dE_{\dot{\theta}[1]}^\mp,\ E_0^\pm=2\sqrt{2} \dE_{\ell[0]}^\pm,\
 \ H_i=r_i \dot{\al}_{i[0]}^\vee,\ H_\ell=-\dot{\theta}^\vee,\ H_0=4\dot\al_{\ell[0]}^\vee, \end{align*}
where $i=1,\cdots,\ell-1$.
  Otherwise, we set
\begin{align}\label{eq:def-Es}
E_i^\pm=r_i \dE_{i[0]}^\pm,\ H_i=r_i \dot{\al}_{i[0]}^\vee,\ E_0^\pm=r\dE_{\dot{\theta}[1]}^\mp,\ H_0=-r\dot{\theta}^\vee_{[0]},\ i=1,\cdots,\ell.
\end{align}
It was proved in \cite[Theorem 8.3]{Kac-book} that
we could identify $\fg$ with $\te{Aff}(\dg,\dot\nu)$ by
\begin{align}\label{identification}
\al_\epsilon^\vee=ra_0^{-1}\rk_2+1\ot H_0,\ e_\epsilon^\pm&=t^{\pm 1} \ot E_\epsilon^\pm,\ \al_i^\vee=1\ot H_i, e_i^\pm=1\ot E_i^\pm,\ i\ne \epsilon,
\end{align}
where  $\epsilon=0, a_0=1$ expect that the GCM $A$ is of type $A_{2\ell}^{(2)}$, in which case $\epsilon=\ell, a_0=2$.
From now on, we will often use the following identifications
\[ \fg=\mathrm{Aff}(\dg,\dot\nu)\quad\te{and}\quad \bar\fg=\bigoplus_{m\in \Z}\C t_2^m\ot \dg_{[m]}\]
without further explanation.

Let $\dot Q=\oplus_{i=1}^n \dot\al_i$ be the root lattice of $\dg$.
Note that $\dot\nu$ induces an automorphism of $\dot{Q}$ such that $\dot{\nu}(\dot{\al}_i)=\dot{\al}_{\dot\nu(i)}$ for
$i=1,2,\dots,n$. Set
\[ \dQ_{[0]}=\{\dot\al_{[0]}=r^{-1}\sum_{p\in \Z_r}\dot{\nu}^p(\dot\al)\mid \dot\al\in \dot{Q}\}\subset \dfh^*.\]
Then the root lattice $Q$ of $\fg$ is equivalent to $\dQ_{[0]}\oplus \Z\delta_2$ and the simple root system $\Pi$ of $\fg$ is equivalent to
\[\{\al_\epsilon=-\dot{\theta}_{[0]}+\delta_2,\ \al_{\ell-\epsilon}=\dal_{\ell[0]},\ \al_i=\dot\al_{i[0]},\ i\ne \epsilon, \ell-\epsilon\}.\]

We extend the normalized bilinear form  $\<\cdot,\cdot\>$  on $\dg$ to a symmetric invariant bilinear form on $\fg$
by letting
\begin{equation*}\begin{split}\label{norform}
\<t_2^{m_1}\ot x+a_1\rk_2, t_2^{m_2}\ot y+a_2\rk_2\>
=\ \ &\delta_{m_1+m_2,0}\,\<x,y\>,
\end{split}\end{equation*}
where $m_1,m_2\in \Z$, $x\in \dg_{[m_1]}$, $y\in \dg_{[m_2]}$ and $a_1,a_2\in \C$.
Since the restriction of $\<\cdot,\cdot\>$ on $\dfh$ is non-degenerate,
we get a non-degenerate bilinear form $(\cdot,\cdot)$ on $\dfh^*$ by duality.
In addition, the bilinear form $(\cdot,\cdot)$ can be extend to a symmetric bilinear form on $Q$ by letting
\begin{align}\label{defbi}(\al+m\delta_2, \beta+n\delta_2)=(\al, \beta),
\end{align}
where $\al,\beta\in \dot{Q}_{[0]}$ and $m,n\in \Z$.

\subsection{Diagram automorphisms}
Throughout this paper, we
let $\mu$ be a permutation of $I$ with order $N$ such that $a_{ij}=a_{\mu(i)\mu(j)}$ for $i,j\in I$.
It is known that $\mu$ induces a {\it diagram automorphism} $\mu$ of $\fg$ such that
\begin{align}\label{defmu}\mu(\al_i^\vee)=\al_{\mu(i)}^\vee,\quad
\mu(e_i^\pm)=e_{\mu(i)}^\pm,\quad i\in I.
\end{align}
This subsection is devoted to an explicit description of the action of $\mu$ on $\fg$.

It is immediate to see that the permutation $\mu$ induces an automorphism
of $Q$ such that $\mu(\delta_2)=\delta_2$.
Recall from \cite[Proposition 8.3]{Kac-book} that
the finite dimensional simple Lie algebra $\dg$ can be generated by the elements
$E_i^+$, $i\in I$ defined in \eqref{eq:def-Es}.
 Then we have that
\begin{lemt}\label{lem:dmu} (a) The action
\begin{align}\label{defdmu} E^+_i\mapsto E^+_{\mu(i)},\quad i\in I,\end{align}
defines (uniquely) an automorphism $\dot\mu$ of $\dg$.

(b) The Cartan subalgebra $\dfh$ of $\dg$ is stable under $\dot\mu$,
and
\begin{align}\label{dmudfh} \dot\mu(\dot\nu(h))=\dot\nu(\dmu(h)),\quad \forall\ h\in \dfh.
\end{align}

(c) There is a homomorphism $\rho_\mu:\dQ\rightarrow \Z$ of abelian groups such that
\begin{align}\label{rhomu} \rho_\mu(\dot\nu(\dot\al))=\rho_\mu(\dot\al),\quad \mu(\dot\al_{[0]})=
\dot\mu(\dot\al)_{[0]}+\rho_\mu(\dot\al)\delta_2,\quad \dot\al\in \dQ.
\end{align}

(d) For  $\dot\al\in \dot\Delta$, $x\in \dg_{\dot\al}$ and $m\in \Z$, we have
 that
\begin{align}\label{dmurhomu}
\dot\mu(x_{[m]})=\dot\mu(x)_{[m+\rho_\mu(\dot\al)]}.
\end{align}
\end{lemt}
\begin{proof}
We first consider the case $\dot\nu={\rm{id}}$. For each $\dal\in \dQ$, write
\bee\label{decmu}\mu(\dal)=\dot{\mu}(\dal)+\rho_\mu(\dal)\delta_2\quad \text{with}\quad \dot{\mu}(\dal)\in \dQ.\eee
Then the map
\[\dot{\mu}:\dQ\rightarrow \dQ,\quad \dal\mapsto \dot{\mu}(\dal)\] is an automorphism of $\dQ$ (with order $N$)
and the map
\[\rho_\mu:\dQ\rightarrow \Z,\quad \dal\mapsto \rho_\mu(\dal)\]
 is a homomorphism of abelian groups.
We define a linear  map $\dot\mu$ on $\dg$ as follows
\[\dot\mu: \dg\rightarrow \dg,\quad \dot{E}_{\dal}^\pm\mapsto \dot\mu(\dot{E}_{\dal}^\pm),\ \ \dal^\vee\mapsto \dot\mu(\dal^\vee),
\,\,\te{for }\dal\in\dot\Delta_+,
\] where $\dot\mu(\dot{E}_{\dal}^\pm)$ are the elements in $\dg_{\pm\dot\mu(\dal)}$ determined by the following equation
\[\mu(1\ot \dot{E}_{\dal}^\pm)=t_2^{\rho_\mu(\pm\dal)}\ot \dot\mu(\dot{E}_{\dal}^\pm).\]
 It is easy to see that $\dot\mu$ is an automorphism of $\dg$ (with order $N$).
Moreover, one can check that the automorphism $\dot\mu$ and the homomorphism $\rho_\mu$ defined above
satisfy all the assertions in the lemma.

Next we consider the case that $\dot\nu\ne {\rm{id}}$. If $\mu={\rm{id}}$, then we only need to take $\dot\mu={\rm{id}}$ and $\rho_\mu=0$.
So we assume further that $\mu$ is nontrivial. Then either
$
X_{n}^{(r)}=A_{2\ell-1}^{(2)},\ \mu=(0,1),
$
or
$
X_{n}^{(r)}=D_{\ell+1}^{(2)},\ \mu=\prod_{0\le i\le \lfloor \frac{\ell-1}{2}\rfloor} (i,l-i).
$
Observe that, if $X_{n}^{(r)}=A_{2\ell-1}^{(2)}$ (resp. $D_{\ell+1}^{(2)}$), then the set
 $\{-\dot\nu(\dot{\theta}), \dal_2,\cdots,\dal_{2\ell-2}, -\dot{\theta}\}$ (resp. $\{\al_{\ell-1},\dal_{\ell-2},\cdots,\dal_{1}, -\dot{\theta}, -\dot{\nu}(\dot{\theta})\}$) is another simple root system of $\dg$.
 Thus, if $X_{n}^{(r)}=A_{2\ell-1}^{(2)}$, then there is an automorphism $\dot\mu$ on $\dg$ given by
\[\dE_1^\pm\mapsto -\dE_{\dot\nu(\dot{\theta})}^\mp,\ \dE_i^\pm\mapsto \dE_i^\pm,\ 2\le i\le 2\ell-2,\  \dE_{2\ell-1}^\pm\mapsto \dE_{\dot{\theta}}^\mp.\]
And if $X_{n}^{(r)}=D_{\ell+1}^{(2)}$, then there is an automorphism $\dot\mu$ on $\dg$ given by
\[\dE_i^{\pm}\mapsto \dE_{\ell-i}^\pm,\ 1\le i\le \ell-1,\ \dE_{\ell}^\pm\mapsto \dE_{\dot{\theta}}^{\mp},\
\dE_{\ell+1}^\pm\mapsto -\dE_{\dot{\nu}(\dot{\theta})}^{\mp}.\]
It is straightforward to check that in both cases the automorphism $\dot{\mu}$ defined above satisfies the properties \eqref{defdmu} and \eqref{dmudfh}.
 This proves the assertions (a) and (b).

For the assertion (c), we define a homomorphism $\rho_\mu:\dot{Q}\rightarrow \Z$  by letting
\begin{align*}\rho_\mu(\dot{\al}_1)=1=\rho_\mu(\dot{\al}_{2\ell-1}),\quad \rho_\mu(\dot{\al}_i)=0,\ 2\le i\le 2\ell-2,\ \te{if}\  X_{n}^{(r)}=A_{2\ell-1}^{(2)},\\
\rho_\mu(\dot{\al}_1)=0,\ 1\le i\le \ell-1,\ \rho_\mu(\dot\al_{\ell})=1=\rho_\mu(\dot\al_{\ell+1}),\ \te{if}\ X_{n}^{(r)}=D_{\ell+1}^{(2)}.
\end{align*}
It is obvious that the property \eqref{rhomu} holds true for all $\dal_i\in \dot{\Pi}$ and hence for all $\dal\in \dQ$.
Finally, it can be checked case by case that, the property \eqref{dmurhomu} holds true for every $x=\dE_i^\pm$, $i=1,2,\dots,n$.
For the general case, we may assume that $\dal=\dal_{i_1}+\cdots+\dal_{i_s}$ and $x=[\dE^+_{i_1},\cdots, [\dE^+_{i_{s-1}}, \dE^+_{i_s}]]$ for some $i_1,\cdots,i_s\in \dot{I}$.
Then
\begin{align*} \dmu(x)&=\dmu(\sum_{k_1,\cdots,k_s\in \Z_r}[\dE^+_{i_1[k_1]},\cdots, [\dE^+_{i_{s-1}[k_{s-1}]}, \dE^+_{i_s[k_s]}]])\\
&=\sum_{k_1,\cdots,k_s\in \Z_r}[\dmu(\dE_{i_1})_{[k_1+\rho_\mu(\dal_{i_1})]},\cdots,
[\dmu(\dE_{i_{s-1}})_{[k_{s-1}+\rho_\mu(\dal_{i_{s-1}})]},\dmu(\dE_{i_s})_{[k_s+\rho_\mu(\dal_{i_s})]}]].
\end{align*}
It implies that
\begin{align*}
& \dmu(x)_{[m+\rho_\mu(\dal)]}\\
=& \sum_{k_1+\cdots+k_s=m}[\dmu(\dE_{i_1})_{[k_1+\rho_\mu(\dal_{i_1})]},\cdots,
[\dmu(\dE_{i_{s-1}})_{(k_{s-1}+\rho_\mu(\dal_{i_{s-1}}))},\dmu(\dE_{i_s})_{(k_s+\rho_\mu(\dal_{i_s}))}]]\\
=&\dmu(\sum_{k_1+\cdots+k_s=m}[\dE^+_{i_1[k_1]},\cdots, [\dE^+_{i_{s-1}[k_{s-1}]}, \dE^+_{i_s[k_s]}]])
=\dmu(x_{[m]}),
\end{align*}
holds true for every $m\in \Z_r$.
This completes the proof of assertion (d).
\end{proof}

Let $\dot\mu$ and $\rho_\mu$ be as in Lemma \ref{lem:dmu}.
Since the bilinear form $\<\cdot,\cdot\>$ is non-degenerated on $\dfh$,
we may and do identify $\dfh$ with its dual space $\dfh^\ast$,
and extend $\rho_\mu$ to a linear functional on $\dfh$ by $\C$-linearity.
The following result is an explicit description of the action of the diagram automorphism $\mu$.

\begin{prpt}\label{charmu} For each $m\in \Z$,  $\dot\al\in \dot{\Delta}\setminus\{0\}$,  $x\in \dg_{\dal}$ and $h\in \dfh$,
we have that
\begin{equation}\begin{split}\label{actmu}
&\mu(t_2^m\ot x_{[m]})= t_2^{m+\rho_\mu(\dot\al)}\ot \dot\mu(x_{[m]}),\quad \mu(\rk_2)= \rk_2,\\
&\mu(t_2^m\ot h_{[m]})= t_2^m\ot \dot\mu(h_{[m]})+\delta_{m,0}\, \rho_\mu(h)\, \rk_2.
\end{split}\end{equation}
\end{prpt}
\begin{proof} Using Lemma \ref{lem:dmu} and the identification \eqref{identification},
one can check that the action given in  \eqref{actmu} defines an automorphism of $\fg$
such that the equation \eqref{defmu} holds, as desired.
\end{proof}

\section{The Lie algebra $\wh\fg[\mu]$ and its MRY presentation}
In this section, we define the twisted toroidal Lie algebra $\wh\fg[\mu]$ and state its Moody-Rao-Yokonuma presentation.
\subsection{The Lie algebra $\wh\fg[\mu]$}
In this subsection, we introduce the definition of the Lie algebra
$\wh\fg[\mu]$.

For $M_1,M_2\in \Z_+$, let $\CK_{M_1,M_2}$ be the $\C$-vector space
spanned by the symbols
\begin{align*}
  t_1^{m_1}t_2^{m_2}\rk_1,\quad t_1^{m_1}t_2^{m_2}\rk_2,\quad m_1\in M_1\Z,\ m_2\in M_2\Z,
\end{align*}
subject to the relation
\begin{align*}
  m_1t_1^{m_1}t_2^{m_2}\rk_1+m_2t_1^{m_1}t_2^{m_2}\rk_2=0.
\end{align*}
We define
\begin{align*}
 \wh\fg= \bigoplus_{m,n\in\Z}\C t_1^mt_2^n\ot\dg_{[n]}\oplus \CK_{1,r}\subset \(\C[t_1^{\pm 1}, t_2^{\pm 1}]\ot \dg\)\oplus \CK_{1,r}
\end{align*}
to be a Lie algebra with Lie bracket given by
\begin{equation}\begin{split}\label{toroidalre}
 &[t_1^{m_1} t_2^{m_2}\ot x, t_1^{n_1} t_2^{n_2}\ot y]\\
 =\ &t_1^{m_1+n_1} t_2^{m_2+n_2}\ot [x,y]+\<x,y\>\, (\sum_{i=1}^2 m_it_1^{m_1+n_1} t_2^{m_2+n_2}\rk_i),
 \end{split}\end{equation}
where $x\in \dg_{[m_2]}$, $y\in \dg_{[n_2]}$, $m_1, m_2,n_1,n_2\in \Z$ and $\CK_{1,r}$ is the center space.
It follows from  \cite{MRY,Sun-uce} that the projective map
\[\psi:\wh\fg\rightarrow \bigoplus_{m,n\in\Z} \C t_1^mt_2^n\ot\dg_{[n]}=\C[t_1,t_1^{-1}]\ot \bar\fg\]
is the universal central extension of the loop algebra $\CL(\bar\fg,{\rm{id}})$ of $\bar\fg$.

For convenience, we view $\C[t_1,t_1^{-1}]\ot \fg$ as a subspace of
$\wh\fg$ in the following way
\begin{equation*}
t_1^{m_1}\ot x= t_1^{m_1}t_2^{m_2}\ot \dot x+at_1^{m_1}\rk_2,
\end{equation*}
for $x=t_2^{m_2}\ot \dot x+a\rk_2\in \fg,\ m_1\in \Z$.
Then it is easy to see that  the Lie algebra $\wh\fg$ is spanned by the elements
\[t_1^{m_1}\ot x, \quad \rk_1,\quad \ t_1^{n_1}t_2^{n_2}\rk_1,\quad x\in \fg,\ m_1,n_1\in \Z,\ n_2\in r\Z^\times.\]
Moreover, the commutator relations among these elements are as follows:
\begin{lemt}\label{lem:commutator} Let $\al, \beta\in \Delta$, $x\in \fg_\al, y\in \fg_\beta$ and $m_1,n_1\in \Z$.
If $\al+\beta\in \Delta^\times\cup \{0\}$, then
\begin{align}\label{commutator1}
[t_1^{m_1}\ot x, t_1^{n_1}\ot y]=t_1^{m_1+n_1}\ot [x,y]+m_1\delta_{m_1,n_1}\<x,y\>\rk_1.\end{align}
If $x=t_2^{m_2}\ot \dot{x}$, $y=t_2^{n_2}\ot \dot{y}$ and $\al+\beta\in \Delta^0\setminus\{0\}$, then
\begin{align}
[t_1^{m_1}\ot x, t_1^{n_1}\ot y]=t_1^{m_1+n_1}\ot [x,y]
+\<\dot{x},\dot{y}\>\frac{m_1n_2-m_2n_1}{m_2+n_2}t_1^{m_1+n_1}t_2^{m_2+n_2}\rk_1.
\end{align}
\end{lemt}

Observe that the Lie algebra $\wh\fg$ is generated by the elements
\begin{align}\label{genhatg}
t_1^m\ot e_i^\pm,\quad t_1^m\ot \al_i^\vee,\quad \rk_1,\quad i\in I,\ m\in \Z.\end{align}
Similar to \eqref{defmu}, the permutation $\mu$ induces an automorphism of $\wh\fg$ as follows.
\begin{lemt}\label{lem:wh-mu-N-period}
The action
\begin{equation}\label{defhatmu}
t_1^m\ot e_i^\pm\mapsto \xi^{-m} t_1^m\ot e_{\mu(i)}^\pm,\quad t_1^m\ot \al_i^\vee\mapsto \xi^{-m}t_1^m\ot \al_{\mu(i)}^\vee,\quad
\rk_1\mapsto \rk_1
\end{equation}
for $i\in I,\ m\in \Z$, defines  an automorphism $\wh\mu$ of $\wh\fg$.
 \end{lemt}
 \begin{proof}
We define a linear transformation $\wh\mu$ on $\wh\fg$ by letting
 \begin{align*}
 t_1^{m_1}\ot x &\mapsto \xi^{-m_1} t_1^{m_1}\ot \mu(x),\\
 t_1^{m_1}\ot h &\mapsto  \xi^{-m_1}(t_1^{m_1}\ot \mu(h)
-\frac{m_1}{m_2}\rho_\mu(\dot{h})t_1^{m_1}t_2^{m_2}\rk_1),\\
\rk_1\mapsto &\rk_1,\quad t_1^{n_1}t_2^{n_2}\rk_1 \mapsto  \xi^{-n_1}t_1^{n_1}t_2^{n_2}\rk_1,
 \end{align*}
where $m_1,n_1\in \Z$, $x\in \fg_\al, \al\in \Delta^\times\cup\{0\}$,
$h=t_2^{m_2}\ot \dot{h}$, $m_2\in \Z^\times$, $n_2\in r\Z^\times$ and $\dot{h}\in \dfh_{[m_2]}$.
Note that $\rho_\mu(\dot{h})\ne 0$ only if $m_2\in r\Z$, and so $\wh\mu$ is well-defined.

By using the explicit action of $\mu$ given in Proposition \ref{charmu} and
 the commutate relations of $\wh\fg$ given in Lemma \ref{lem:commutator}, one can easily verify that
 the map $\wh\mu$ is an automorphism of $\wh\fg$.
 Moreover, it is obvious that the actions of $\wh\mu$ on those generators  in \eqref{genhatg}  coincide with that in \eqref{defhatmu}.
 This completes the proof.
 \end{proof}

We define $\wh\fg[\mu]$ to be the subalgebra of $\wh\fg$ fixed by  $\wh\mu$.
Recall from Introduction that $\bar\mu$ is the automorphism of $\bar\fg$ induced from $\mu$, and that
$\CL(\bar\fg,\bar\mu)$ is the twisted loop algebra of $\bar\fg$ related to $\bar\mu$.
Note that  $\CL(\bar\fg,\bar\mu)$ is the subalgebra of $\CL(\bar\fg,{\rm{id}})$ fixed by the automorphism
\[\xi^{-\rd_1}\ot \bar\mu:\CL(\bar\fg,{\rm{id}})\rightarrow \CL(\bar\fg,{\rm{id}}),\quad t_1^m\ot x\mapsto \xi^{-m} t_1^m\ot \bar\mu(x),\ m\in \Z, x\in \bar\fg.\]
It follows from \eqref{defhatmu} that
\begin{align}\label{hatmucover}\psi\circ\wh{\mu}=(\xi^{-\rd_1}\ot \bar\mu)\circ \psi.\end{align}
Thus, by taking the restriction of $\psi$ on $\wh\fg[\mu]$, one gets a
Lie algebra homomorphism
 \[\psi_\mu=\psi|_{\wh\fg[\mu]}:\wh\fg[\mu]\rightarrow \CL(\bar\fg,\bar\mu).\]
The following theorem is the first main result of this paper, whose proof
 will be presented in Section \ref{sec:proof-uce}.
\begin{thm}\label{main1} The Lie algebra homomorphism  $\psi_\mu:\wh\fg[\mu]\rightarrow \CL(\bar\fg,\bar\mu)$
is the universal central extension of the twisted loop algebra $\CL(\bar\fg,\bar\mu)$.
\end{thm}

%

\subsection{The MRY presentation}
Here we state an MRY presentation for $\wh\fg[\mu]$.
Throughout this subsection, we may always assume that $\mu$ is non-transitive.
Observe  that
a diagram automorphism on $\fg$ is transitive if and only if $\fg$ is of type $A_\ell^{(1)}$ $(\ell\ge 1)$,
and the diagram automorphism is an order $\ell+1$ rotation of the Dynkin diagram.

We first introduce some notations. Set $V=\mathbb R\ot_\Z Q$ and extend $(\cdot,\cdot)$ (see \eqref{defbi}) to a bilinear form on $V$ by $\R$-linearity.
For  $i,j\in I$, we set
\[\check{\al}_i=\frac 1N\sum_{k\in \Z_N}\al_{\mu^k(i)}\quad\te{and}\quad \check{a}_{ij}=2\frac{(\check{\al}_i,\check{\al}_j)}{(\check{\al}_i,\check{\al}_i)}.\]
We fix a representative subset of $I$ as follows
\[ \check{I}=\{ i\in I\,|\,\mu^k(i)\geq i\ \text{for}\ k\in \Z_N\}. \]
It was proved in \cite[Proposition 12.1.10]{ABP} (see also \cite{FSS}) that  the folded matrix
\[\check{A}=\left( \check{a}_{ij}\right)_{i,j\in\check{I}},\quad\]
 of the GCM $A$ associated to $\mu$ is also a GCM of affine type.

For $i\in I$, we denote by $\mathcal{O}(i)\subset I$ the orbit containing $i$ under the action of the group $\la \mu\ra$.
The following result was proved in
\cite[Lemma 12.1.5]{ABP}.
\begin{lemt}\label{lem:defsi}
 For each $i\in I$,  exactly one of the following holds
\begin{enumerate}
\item[(a)] The elements $\alpha_p, p\in \mathcal{O}(i)$ are pairwise orthogonal;
\item[(b)] $\mathcal{O}(i)=\{i,\mu(i)\}$ and
$a_{i\mu(i)}=-1=a_{\mu(i)i}$.
\end{enumerate}
\end{lemt}

As in \cite{ABP}, for $i\in I$, we set
\begin{equation*}
s_i=\begin{cases}1,\ \text{if (a) holds in Lemma \ref{lem:defsi}};\\
2,\ \text{if (b) holds in Lemma \ref{lem:defsi}}.\end{cases}\end{equation*}
Now we introduce the following definition.
\begin{dfnt}
Define $\mathcal{M}(\fg,\mu)$ to be the Lie algebra generated by the elements
\begin{eqnarray}\label{eq:genset}
  h_{i,m},\quad x_{i,m}^\pm,\quad c,\quad i\in I,\ m\in \mathbb Z,
\end{eqnarray}
and subject to the relations
\begin{eqnarray*}
&&\text{(T0) }h_{\mu (i),m}=\xi^m h_{i,m},\ \ \ \
x^\pm_{\mu (i),m}=\xi^m x^\pm_{i,m},\\
&&\text{(T1) }[c, h_{i,n}]=0=[c, x^{\pm}_{i,n}],\\
&&\text{(T2) }[h_{i,m},h_{j,n}]
=\sum_{k\in \Z_N}mN\<\al_i^\vee,\al_{\mu^k(j)}^\vee\>\delta_{m+n,0}m\xi^{km}c,\\
&&\text{(T3) }[h_{i,m},x^\pm_{j,n}]
=\pm \sum_{k\in \Z_N}a_{i\mu^k(j)} x^\pm_{j,m+n}\xi^{km},\\
&&\text{(T4) }[x^+_{i,m},x^-_{j,n}]
=\sum_{k\in \Z_N}\delta_{i,\mu^k (j)}\left(h_{j,m+n}+\frac{mN\<\al_i^\vee,\al_i^\vee\>}{2}\delta_{m+n,0}c\right)\xi^{km},\\
&&\text{(T5) }\left(\mathrm{ad}\, x_{i,0}^\pm\right)^{1-\check{a}_{ij}}
\left(x_{j,m}^\pm\right)=0,\ \text{if}\ \check{a}_{ij}\le 0,\\
&&\text{(T6) }[x_{i,m_1}^\pm,\cdots,[x^\pm_{i,m_{s_i}}, x_{i,m_{s_i+1}}^\pm]]=0.
\end{eqnarray*}
\end{dfnt}

In view of \eqref{genhatg} and \eqref{defhatmu}, we know that the Lie algebra $\wh\fg[\mu]$ is generated by the following elements
\begin{align} \label{genwhfgmu} t_1^m\ot e_{i(m)}^\pm,\quad t_1^m\ot \al_{i(m)}^\vee,\quad \rk_1,\quad i\in I,\ m\in \Z,
\end{align}
where  $x_{(m)}=\sum_{p\in \Z_N}\xi^{-pm}\mu^p(x)$ for $x\in \fg$.
The following theorem is the second main result of this paper, whose proof will be presented in Section \ref{sec:proof-MRY}.
\begin{thm} \label{main2}
The assignment
\begin{eqnarray*}  c\mapsto \rk_1,\ h_{i,m}\mapsto t_1^m\ot \al_{i(m)}^\vee,\ x_{i,m}^\pm\mapsto t_1^m \ot e^\pm_{i(m)},\ i\in I, m\in \Z
\end{eqnarray*}
 determines a Lie algebra isomorphism from $\mathcal{M}(\fg,\mu)$  to $\wh\fg[\mu]$.
\end{thm}

When $\fg$ is of untwisted type and $\mu=1$, Theorem \ref{main2} was proved in \cite{MRY}.

\section{Proof of Theorem \ref{main1}}\label{sec:proof-uce}
\subsection{Multiloop algebras}
We start by recalling the definition of multiloop algebras (see \cite{ABFP}).
Let $\mathfrak k$ be an arbitrary Lie algebra, and let $\sigma_1$, $\sigma_2,\dots,\sigma_s$ be
 pairwise commuting automorphisms on $\mathfrak k$.
From now on, we denote by
\[\mathfrak k^{\,\sigma_1,\sigma_2,\dots,\sigma_s}\]
the fixed point subalgebra of $\mathfrak k$ under the automorphisms $\sigma_1,\sigma_2,\dots,\sigma_s$.
Suppose further that each automorphism $\sigma_i$ has a finite period $M_i$, i.e. $\sigma^{M_i}=1$, $i=1, \ldots, s$.
The {\it multiloop algebra} associated to $\fk$, $\sigma_1,\sigma_2,\dots,\sigma_s$ is by definition the  following
subalgebra of  $\C[t_1^{\pm 1},t_2^{\pm 1},\dots,t_s^{\pm 1}]\ot \fk $:
\baa \mathcal L_{M_1,M_2,\cdots,M_s}(\mathfrak k,\sigma_1,\sigma_2,\dots,\sigma_s)=\bigoplus_{m_1,m_2,\dots,m_s\in \Z}
\C t_1^{m_1}t_2^{m_2}\cdots t_s^{m_s}\ot \mathfrak k_{(m_1,m_2,\dots,m_s)},\eaa
where
\[\mathfrak k_{(m_1,m_2,\dots,m_s)}=\{x\in \mathfrak k\mid \sigma_i(x)=\xi_{M_i}^{m_i}\ x,\ i=1,2,\dots,s\},\]
and when each $M_i$ is the order of $\sigma_i$ we often write
\[\mathcal L(\mathfrak k,\sigma_1,\sigma_2,\dots,\sigma_s)=\mathcal L_{M_1,M_2,\dots,M_r}(\mathfrak k,\sigma_1,\sigma_2,\dots,\sigma_s).\]

Let $\sigma$ be an automorphism of $\mathfrak k$, and $(c_1,c_2,\dots,c_s)$ an $s$-tuple in $(\mathbb C^\times)^s$.
Let 
\[c_1^{-\rd_1}\ot c_2^{-\rd_2}\ot \cdots\ot c_s^{-\rd_s}\ot \sigma\]
be 
the automorphism of $\C[t_1^{\pm 1},t_2^{\pm 1},\dots,t_s^{\pm 1}]\ot \fk$ defined by:
\[
t_1^{m_1}t_2^{m_2}\cdots t_s^{m_s}\ot x\mapsto c_1^{-m_1}c_2^{-m_2}\cdots c_s^{-m_s} t_1^{m_1}t_2^{m_2}\cdots t_s^{m_s}\ot \sigma(x),\]
where $x\in \fk$ and $m_i\in \Z$.
It is obvious that the multiloop algebra $\mathcal L_{M_1,\dots,M_s}(\mathfrak k,\sigma_1,\dots,\sigma_s)$
is the  subalgebra of
$\mathbb C[t_1^{\pm1},t_2^{\pm 1},\dots, t_s^{\pm 1}]\ot \fk$
fixed by the  following commuting automorphisms
\[\xi_{M_1}^{-\rd_1}\ot 1\ot \cdots\ot 1\ot \sigma_1,\ 1\ot \xi_{M_2}^{-\rd_2}\ot \cdots\ot 1\ot \sigma_2,\ \dots,\ 1\ot \cdots\ot 1\ot \xi_{M_s}^{-\rd_s}\ot \sigma_s.\]

\subsection{The functor $\uce$}
In this subsection, we recall the endofunctor $\uce$ on the category of Lie algebras introduced in \cite{N-uce-functor}.
Let $\mathfrak k$ be an arbitrary Lie algebra, and $B$ the subspace of $\fk\otimes \fk$ spanned by all elements of
the form
\begin{eqnarray*}
x\otimes y+y\otimes x\ \text{and }
x\otimes [y,z]+y\otimes [z,x]+z\otimes [x,y],\quad x,y,z\in \fk.
\end{eqnarray*}
We  define
\begin{eqnarray*}
\uce(\fk)=\fk\otimes \fk/B
\end{eqnarray*}
to be a Lie algebra with Lie bracket
\begin{eqnarray*}
[x\otimes x',y\otimes y']_{\uce(\fk)}=[x,x']\otimes [y,y']+B.
\end{eqnarray*}
Then we have the following well-defined Lie algebra homomorphism
\begin{eqnarray*}
\mathfrak{u}_{\fk}:\uce(\fk)\longrightarrow [\fk,\fk]\subset \fk,\quad x\otimes y\mapsto [x,y],
\end{eqnarray*}
which is in fact a central extension of $[\fk,\fk]$.

Let $f:\fk\rightarrow \fk_0$ be a homomorphism of Lie algebras. Then the map
\begin{eqnarray*}
\uce(f):&&\uce(\fk)\longrightarrow \uce(\fk_0)\nonumber\\
&&x\otimes y\mapsto f(x)\otimes f(y),
\end{eqnarray*}
is also a Lie algebra homomorphism. Note that $\uce$ is a covariant functor. Therefore, if $f$ is an isomorphism, then so is $\uce(f)$.

 We say that a homomorphism $\hat{f}:\uce(\fk)\rightarrow \uce(\fk_0)$ {\it covers}  $f: \fk\rightarrow \fk_0$ if
 \[\mathfrak u_{\fk_0}\circ \hat{f}=f\circ \mathfrak u_\fk.\]
The following results were proved in \cite{N-uce-functor}.
\begin{prpt}\label{uce} Let $\fk$ be a perfect Lie algebra. Then
\begin{enumerate}
\item[(a)] the map $\mathfrak u_\fk:\uce(\fk)\rightarrow \fk$ is the universal central extension of $\fk$, and $\ker(\mathfrak u_\fk)$ is the center of $\uce(\fk)$
when $\fk$ is centerless;
\item[(b)] for any homomorphism $f:\fk\rightarrow \fk_0$ of Lie algebras, the map $\uce(f)$ is the unique homomorphism from $\uce(\fk)$
to $\uce(\fk_0)$ that covers $f$.
\end{enumerate}

\end{prpt}

We also record the following trivial result as a lemma that will be used later on.
\begin{lemt}\label{conj}Let
$\sigma_1,\dots,\sigma_s$ and $\tau_1,\dots,\tau_s$ be pairwise commuting automorphisms of Lie algebras $\mathfrak k$ and $\fk_0$, respectively. Assume that there is a homomorphism $\gamma:\mathfrak k\rightarrow \mathfrak k_0$ such that $\gamma\circ \sigma_i=\tau_i\circ \gamma$ for each $i=1,\dots,s$. Then one has that
\begin{enumerate}
\item[(a)] if the map $\uce(\gamma)$ is injective, then
\[\uce(\gamma)(\uce(\mathfrak k)^{\uce(\sigma_1),\dots,\uce(\sigma_s)})= \uce(\mathfrak k_0)^
{\uce(\tau_1),\dots,\uce(\tau_s)}\cap \mathrm{im}(\uce(\gamma));\]
\item[(b)] if the map $\gamma$ is an isomorphism, then
\[\uce(\gamma):\uce(\mathfrak k)^{\uce(\sigma_1),\dots,\uce(\sigma_s)}\cong \uce(\mathfrak k_0)^
{\uce(\tau_1),\dots,\uce(\tau_s)}.\]
\end{enumerate}
\end{lemt}

Suppose now that $\sigma_1$ and $\sigma_2$ are two commuting automorphisms of $\dg$ with periods $M_1$
and $M_2$, respectively.
We define
\[\wh{\CL}_{M_1,M_2}(\dg,\sigma_1,\sigma_2)=\CL_{M_1,M_2}(\dg,\sigma_1,\sigma_2)\oplus \CK_{M_1,M_2}\] to
be the Lie algebra with Lie bracket as in \eqref{toroidalre}.
In particular,  we have that
$\wh\fg=\wh{\CL}_{1,r}(\dg,{\rm{id}},\dot\nu)$.
It was proved  in \cite{Sun-uce} that $\wh{\CL}_{M_1,M_2}(\dg,\sigma_1,\sigma_2)$ is the universal central extension
of $\CL_{M_1,M_2}(\dg,\sigma_1,\sigma_2)$.
For convenience, when $M_i$ is the order of $\sigma_i$ for $i=1,2$, we also write
$\wh{\CL}(\dg,\sigma_1,\sigma_2)=\wh{\CL}_{M_1,M_2}(\dg,\sigma_1,\sigma_2)$.

For $\sigma\in \mathrm{Aut}(\dg)$ and $c_1,c_2\in \C^\times$, one can easily verify that the assignment
\begin{align*}
t_1^{m_1} t_2^{m_2}\ot x&\mapsto c_1^{-m_1}c_2^{-m_2} t_1^{m_1}t_2^{m_2}\ot \sigma(x),\ x\in \dg,\ m_1,m_s\in \Z,\\
t_1^{m_1}t_2^{m_2}\rk_i&\mapsto c_1^{-m_1}c_2^{-m_2} t_1^{m_1}t_2^{m_2}\rk_i,\ i=1,2,\end{align*}
determines an automorphism on  $\wh\CL(\dg,{\rm{id}},{\rm{id}})=\uce(\CL(\dg,{\rm{id}},{\rm{id}}))$.
Note that this automorphism covers $c_1^{-\rd_1}\ot c_2^{-\rd_2}\ot \sigma$,
and hence coincides with $\uce(c_1^{-\rd_1}\ot  c_2^{-\rd_2}\ot \sigma)$ (Proposition \ref{uce} (b)).
Using this, it is easy to see that
\begin{align*}
\wh{\CL}_{M_1,M_2}(\dg,\sigma_1,\sigma_2)=(\wh\CL(\dg,{\rm{id}},{\rm{id}}))^{\uce(\xi_{M_1}^{-\rd_1}\ot  1^{-\rd_2}\ot \sigma_1),
\uce(1^{-\rd_1}\ot  \xi_{M_2}^{-\rd_2}\ot \sigma_2)}.
\end{align*}
In other words, we have the following isomorphism
\begin{equation}\begin{split}\label{split}
&\uce\(\CL(\dg,{\rm{id}},{\rm{id}})^{\xi_{M_1}^{-\rd_1}\ot  1^{-\rd_2}\ot \sigma_1,
1^{-\rd_1}\ot  \xi_{M_2}^{-\rd_2}\ot \sigma_2}\)\\
\cong\ &\uce(\CL(\dg,{\rm{id}},{\rm{id}}))^{\uce(\xi_{M_1}^{-\rd_1}\ot  1^{-\rd_2}\ot \sigma_1),
\uce(1^{-\rd_1}\ot  \xi_{M_2}^{-\rd_2}\ot \sigma_2)}.\end{split}\end{equation}

\subsection{Automorphism groups}
In this subsection we collect some basics on the automorphism group of $\fg$,
one may consult \cite[Section 6]{ABP} for details.
Let $\Aut(A)$ be the group of diagram automorphisms of $\fg$.
Define the outer automorphism group of $\fg$  to be
\begin{equation*}
\mathrm{Out}(A)=
\la \omega \ra\times \Aut(A),
\end{equation*} where $\omega$ is the Chevalley involution of $\fg$.

Let  $\mathrm{Hom}(Q,\C^\times)$ denote the set of group homomorphisms from $Q$ to $\C^\times$, which is
 viewed as a group under pointwise multiplication.
The group $\mathrm{Hom}(Q,\C^\times)$ can be identified as
 a subgroup of $\Aut(\fg)$ in the following way:
\begin{align}\label{QinAutg} \mathrm{Hom}(Q,\C^\times)\hookrightarrow \Aut(\fg),\quad
\rho\mapsto (x\mapsto \rho(\al)\,x),\quad x\in \fg_\al,\ \al\in \Delta.
\end{align}
Define the inner automorphism group of $\fg$ to be
\[\Aut^0(\fg)=\la \mathrm{exp}(\mathrm{ad} x_\alpha)\mid  \al\in \Delta^\times\ra\cdot\mathrm{Hom}(Q,\C^\times).\]

Consider now the group homomorphism
\[\bar\chi:\Aut(\fg)\rightarrow \Aut(\bar\fg),\]
 where
$\bar\chi(\tau)=\bar\tau$ is the automorphism of $\bar\fg$ induced from $\tau$.
Note that the restriction of $\bar\chi$ on $\mathrm{Out}(A)$ and $\mathrm{Hom}(Q,\C^\times)$ are both injective.
Thus we may view them as subgroups of $\Aut(\bar\fg)$.
The following statements were proved in \cite[Proposition 6.1.5 and Proposition 6.1.8]{ABP}.
\begin{prpt}\label{affaut}
The homomorphism $\bar\chi$ is an isomorphism.
Furthermore
\[\Aut(\fg)=\Aut^0(\fg)\rtimes \mathrm{Out}(A),\quad \Aut(\bar\fg)=\Aut^0(\bar\fg)\rtimes \mathrm{Out}(A),\]
where $\Aut^0(\bar\fg)=\bar\chi(\Aut^0(\fg))$.
\end{prpt}

By Proposition \ref{affaut}, we have the following projections
\[p:\Aut(\fg)\rightarrow \mathrm{Out}(A)\quad \text{and}\quad \bar{p}:\Aut(\bar\fg)\rightarrow \mathrm{Out}(A)\]
such that $\bar{p}\circ \bar\chi=p$.
An automorphism $\sigma$ of $\fg$ (resp. $\bar\fg$) is said to be of the {\it first kind} if  $p(\sigma)$ (resp. $\bar{p}(\sigma)$)   lies in $\Aut(A)$.
Otherwise, we say that $\sigma$ is of the {\it second kind}.

\subsection{Universal central extensions}
This subsection is devoted to a proof of the following theorem.
\begin{thm}\label{central}
Let $\bar\eta$ be an  automorphism of $\bar\fg$ of the first kind with
period $M$. Then the Lie algebra
$\uce(\CL(\bar\fg,{\rm{id}}))^{\uce(\xi_M^{-\rd_1}\ot \bar\eta)}$ is the  universal central extension of the  loop algebra
$\CL_M(\bar\fg,\bar\eta)=\CL(\bar\fg,{\rm{id}})^{\xi_M^{-\rd_1}\ot \bar\eta}$.
\end{thm}

Recall that the automorphism $\wh{\mu}$ of $\wh\fg=\uce(\CL(\bar\fg,{\rm{id}}))$ covers the automorphism
$\xi^{-\rd_1}\ot \bar\mu$ of $\CL(\bar\fg,{\rm{id}})$ (see \eqref{hatmucover}), and so coincides with $\uce(\xi^{-\rd_1}\ot \bar\mu)$ (Proposition \ref{uce}(b)).
Thus, Theorem \ref{main1} is just a special case of Theorem \ref{central}.


We first prove some technique lemmas.
 Let $\bar\sigma$ be an automorphism of $\bar\fg$ with period $M$. It is known that the twisted loop algebra of $\bar\fg$ related to
 $\bar\sigma$ is independent from the choice of its periods (\cite[Lemma 2.3]{ABP-covering2}).
In the following, we extend this result to their universal central extensions.
\begin{lemt}\label{lem:period-uce}
Let $\bar\sigma$ be an automorphism of $\bar\fg$ of finite period,
and $M,M'$ two periods of $\bar\sigma$. Then
\begin{equation}
\uce(\CL(\bar\fg,{\rm{id}}))^{\uce( \xi_M^{-\rd_1}\ot \bar\sigma)}\cong
\uce(\CL(\bar\fg,{\rm{id}}))^{\uce( \xi_{M'}^{-\rd_1}\ot \bar\sigma)}.\end{equation}
\end{lemt}
\begin{proof} We may (and do) assume that $M'=bM$ for some $b\in \Z_+$.
Consider the natural imbedding
\[i_b: \CL(\bar\fg,{\rm{id}})\rightarrow \CL(\bar\fg,{\rm{id}}),\quad t_1^m\ot x\mapsto t_1^{bm}\ot x,\]
where $m\in\Z$ and $x\in \bar\fg$. It is clear that the image of $i_b$ is the Lie algebra $\mathcal L_b(\bar\fg,{\rm{id}})=\mathcal L_{b,r}(\dg,{\rm{id}},\dot\nu)$ and
that
\begin{align}\label{period11}(\xi_{M'}^{-\rd_1}\ot \bar\sigma)\circ i_b=
i_b\circ (\xi_{M}^{-\rd_1}\ot \bar\sigma).
\end{align}
Using Proposition \ref{uce} (b), it is easy to see that the action of $\uce(i_b)$ on the center of  $\uce(\CL(\bar\fg,{\rm{id}}))$ $=
 \uce\({\mathcal L}_{1,r}(\dg,{\rm{id}},\dot\nu)\)$ is given by
 \begin{align*}
 t_1^{m_1}t_2^{m_2}\rk_i\mapsto t_1^{bm_1}t_2^{m_2}\rk_i,\quad i=1,2,\ m_1\in \Z,\ m_2\in r\Z.
 \end{align*}
 This implies  that
\begin{align}\label{period12}
\te{the map}\ \uce(i_b)\ \te{is injective,}\end{align} and that
 \begin{equation}
 \begin{split}\label{period13}
 &\mathrm{im}(\uce(i_b))=\uce\({\mathcal L}_{b,r}(\dg,{\rm{id}},\dot\nu)\)\\
 =\ &\uce(\mathcal L(\dg,{\rm{id}},{\rm{id}}))^{\uce(\xi_b^{-\rd_1}\ot 1^{-\rd_2}\ot 1),\uce(1^{-\rd_1}\ot \xi_r^{-\rd_2}\ot \dot\nu)}\\
 =\ &(\uce(\mathcal L(\dg,{\rm{id}},{\rm{id}}))^{\uce(1^{-\rd_1}\ot \xi_r^{-\rd_2}\ot \dot\nu)})^{\uce(\xi_b^{-\rd_1}\ot 1^{-\rd_2}\ot 1)}\\
 =\ &\uce(\CL(\bar\fg,{\rm{id}}))^{\uce(\xi_b^{-\rd_1}\ot 1)}.\end{split}\end{equation}
Note that we also have
\begin{equation*}\begin{split}
&\uce(\CL(\bar\fg,{\rm{id}}))^{\uce(\xi_{M'}^{-\rd_1}\ot \bar\sigma)}\subset
\uce(\CL(\bar\fg,{\rm{id}}))^{(\uce(\xi_{M'}^{-\rd_1}\ot \bar\sigma))^M}\\
=\ &\uce(\CL(\bar\fg,{\rm{id}}))^{\uce((\xi_{M'}^{-\rd_1}\ot \bar\sigma)^M)}
= \uce(\CL(\bar\fg,{\rm{id}}))^{\uce(\xi_b^{-\rd_1}\ot 1)}.
\end{split}\end{equation*}
This together with \eqref{period13} gives that
\begin{align}\label{period14}\mathrm{im}(\uce(i_b))\cap \uce(\CL(\bar\fg,{\rm{id}}))^{\uce(\xi_{M'}^{-\rd_1}\ot \bar\sigma)}=
\uce(\CL(\bar\fg,{\rm{id}}))^{\uce(\xi_{M'}^{-\rd_1}\ot \bar\sigma)}.
\end{align}
Now  the assertion is implied by \eqref{period11}, \eqref{period12},
\eqref{period14} and Lemma \ref{conj} (a).
\end{proof}

Let $\bar\sigma$ be an automorphism of $\bar\fg$ with period $M$.
Now $\bar\fg=\CL(\dg,\dot\nu)$  itself is a twisted loop algebra and so  is independent from the choice of the period of $\dot\nu$.
Namely, if $M'$ is another period of $\dot\nu$, then one has the natural isomorphism
$\bar\fg\cong \CL_{M'}(\dg,\dot\nu)$. Via this isomorphism, $\bar\sigma$ induces an automorphism,
say $\bar\sigma'$, of $\CL_{M'}(\dg,\dot\nu)$ with period $M$.
Similar to Lemma \ref{lem:period-uce}, we have that
\begin{lemt}\label{lem:period-uce2} Let $\bar\sigma,M,M'$ and $\bar\sigma'$ be as above. Then one has that
\begin{align}
\uce(\CL(\bar\fg,{\rm{id}}))^{\uce(\xi_M^{-\rd_1}\ot \bar\sigma)}\cong
\uce(\CL(\CL_{M'}(\dg,\dot\nu),{\rm{id}}))^{\uce(\xi_M^{-\rd_1}\ot \bar\sigma')}.
\end{align}
\end{lemt}
\begin{proof}
 Set $b=M'/r$ and define the embedding
\[j_b: \bar\fg=\mathcal L(\dg,\dot\nu)\rightarrow \mathcal L(\dg,{\rm{id}}),\quad t_2^{m_2}\ot x\mapsto t_2^{bm_2}\ot x,\ m_2\in \Z, x\in \dg.\]
Then the image of $j_b$ is the Lie algebra $\mathcal L_{M'}(\dg,\dot\nu)$ and
\begin{align}\label{eq:def-bar-tau}j_b\circ\bar\sigma=\bar\sigma'\circ j_b.\end{align}
Moreover,
the
action of $\uce\(1^{-\rd_1}\ot j_b\)$ on the center of
$\uce\(\CL\(\bar\fg,{\rm{id}}\)\)$
is  given by
\begin{align*}
  t_1^{m_1}t_2^{m_2}\rk_i\mapsto t_1^{m_1}t_2^{bm_2}\rk_i,\quad i=1,2,\,m_1\in\Z, m_2\in r\Z.
\end{align*}
This implies that
\begin{align}\label{eq:inj-j}
\te{the map }\uce(1^{-\rd_1}\ot j_b)\,\te{is injective}
\end{align}
and that
\begin{align}\label{eq:im-j}
\mathrm{im}(\uce(1^{-\rd_1}\ot j_b))=\uce\(\CL_{1,M'}\(\dg,{\rm{id}},\dot\nu\)\)
=\uce\(\CL\(\CL_{M'}(\dg,\dot\nu),{\rm{id}}\)\).
\end{align}
Then  the lemma follows from \eqref{eq:def-bar-tau},
\eqref{eq:inj-j}, \eqref{eq:im-j} and  Lemma \ref{conj} (a).
\end{proof}

Using Lemma \ref{lem:period-uce}, we have the following result.
\begin{lemt}\label{prop:autdec} Let $\bar\sigma$ be an automorphism of $\bar\fg$ with period $M$. Then
\begin{align}
\uce(\CL(\bar\fg,{\rm{id}}))^{\uce(\xi_M^{-\rd_1}\ot \bar\sigma)}\cong
\uce(\CL(\bar\fg,{\rm{id}}))^{\uce(\xi_M^{-\rd_1}\ot \bar{p}(\bar\sigma))}.
\end{align}
\end{lemt}
\begin{proof}
Recall the isomorphism $\bar\chi:\Aut(\fg)\rightarrow \Aut(\bar\fg)$ given in Proposition \ref{affaut}.
Then we may choose an automorphism $\sigma$ of $\fg$
such that $\sigma^M={\rm{id}}$ and $\bar\chi(\sigma)=\bar\sigma$. This together with  \cite[Lemma 4.31]{KW} gives that there exists a
$\rho\in \mathrm{Hom}(Q,\C^\times)$ such that
\[\rho\, \bar{p}(\bar{\sigma})=\bar{p}(\bar{\sigma})\,\rho,\quad \rho^M={\rm{id}}\quad \te{and}\quad \bar\sigma\ \te{is conjugate to}\ \bar{p}(\bar{\sigma})\rho.\]
Note that the automorphisms $\rho$ and $\bar{p}(\bar{\sigma})$ of $\bar\fg$ satisfy all
the assumptions stated in  \cite[Theorem 5.1]{ABP-covering2}.
Then it follows from \cite[(5.3)]{ABP-covering2} that
the automorphism $\xi_M^{-\rd_1}\ot \rho\bar{p}(\bar\sigma)$ is conjugate to
$\xi_{M^2}^{-\rd_1}\ot \bar{p}(\bar\sigma)$.
This together with  Lemma \ref{conj} (b) and Lemma \ref{lem:period-uce} gives that
\begin{align*}
\uce(\CL(\bar\fg,{\rm{id}}))^{\uce(\xi_M^{-\rd_1}\ot \bar\sigma)}
&\cong \uce(\CL(\bar\fg,{\rm{id}}))^{\uce(\xi_M^{-\rd_1}\ot \rho\bar{p}(\bar\sigma))}\\
\cong \uce(\CL(\bar\fg,{\rm{id}}))^{\uce(\xi_{M^2}^{-\rd_1}\ot \bar{p}(\bar\sigma))}
&\cong \uce(\CL(\bar\fg,{\rm{id}}))^{\uce(\xi_{M}^{-\rd_1}\ot \bar{p}(\bar\sigma))}.
\end{align*}
Therefore, we complete the proof.
\end{proof}

Let $\mathrm{Hom}(\dot{Q},\C^\times)$ be the set of group homomorphisms from $\dot{Q}$ to $\C^\times$.
Similar to \eqref{QinAutg}, we may (and do) view $\mathrm{Hom}(\dot{Q},\C^\times)$ as a subgroup of $\Aut(\dg)$.
From now on, let $\bar\eta$ be as in Theorem \ref{central}.
The following characterization of $\mathcal L_M(\bar{\fg},\bar\eta)$ plays a key role in the proof of Theorem \ref{central}.

\begin{lemt}\label{iso1} There exist finite order automorphisms $\dot\rho$ and $\dot\tau$ of $\dg$  such that
\[\dot\rho\in \mathrm{Hom}(\dot{Q},\C^\times),\ \dot\nu\dot\rho=\dot\rho\dot\nu,\ (\dot\nu\dot\rho)\dot\tau=\dot\tau(\dot\nu\dot\rho)\ \te{and}\ \mathcal L_{M_1,M_2}(\dg,\dot\tau,\dot\nu\dot\rho)\cong \mathcal L_M(\bar{\fg},\bar\eta),\]
where $M_1$ and $M_2$ are some periods of $\dot\tau$ and $\dot\nu\dot\rho$, respectively.
\end{lemt}
\begin{proof}By \cite[Theorem 10.1.1]{ABP}, there exist finite order automorphisms $\dot\tau$ and $\dot\sigma$ such that
$\CL(\bar\fg,\bar\eta)\cong \CL(\dg,\dot\tau,\dot\sigma)$. Up to conjugation, we may assume that $\dot\sigma$
is of the form $\dot\rho\,\dot\vartheta$, where $\dot\rho\in \mathrm{Hom}(\dot{Q},\C^\times)$ and
$\dot\vartheta$ is a diagram automorphism of $\dg$ such that $\dot\rho\,\dot\vartheta=\dot\vartheta\,\dot\rho$.
If $\fg$ is of untwisted type, then it follows from the proof of \cite[Theorem 10.1.1]{ABP} that one may take $\dot\vartheta={\rm{id}}=\dot\nu$.
If $\fg$ is of twisted type,  then  by comparing the classification results (the relative and absolute types) given in  \cite[Tables 3]{ABP} and
\cite[Table 9.2.4]{GP-torsors}, we find out that the diagram automorphism $\dot\vartheta$ can also be taken to be $\dot\nu$.
\end{proof}

Notice that the automorphisms $\dot\rho$ and $\dot\nu$ satisfy the assumptions given in
\cite[Theorem 5.1]{ABP-covering2}. Thus, there is an automorphism $\varphi$ of $\mathcal L(\dg,{\rm{id}})$ such that
\begin{align}\label{autpsi}
\varphi\circ (\xi_{M_2}^{-\rd_2}\ot \dot\nu\dot\rho)\circ \varphi^{-1}=\xi_{M_2^2}^{-\rd_2}\ot \dot\nu.\end{align}
Denote by $\tau'$ the automorphism
\[\varphi\circ (1^{-\rd_2}\ot \dot\tau)\circ \varphi^{-1}\]
 of $\mathcal L(\dg,{\rm{id}})$. Then $\tau'$ commutes with the automorphism $\xi_{M_2^2}^{-\rd_2}\ot \dot\nu$, and hence preserves the Lie
algebra $\mathcal L_{M_2^2}(\dg,\dot\nu)$. Write $\tau''$ for the restriction of $\tau'$ on $\mathcal L_{M_2^2}(\dg,\dot\nu)$, and $\bar\tau$ for the automorphism of $\bar\fg$ induced from $\tau''$ via the isomorphism $\bar\fg\cong \CL_{M_2^2}(\dg,\dot\nu)$.
So by definition we have that
\begin{equation}\begin{split}\label{eq:rhonutau}
&\CL(\dg,{\rm{id}},{\rm{id}})^{\xi_{M_1}^{-\rd_1}\ot 1^{-\rd_2}\ot \dot\tau,1^{-\rd_1}\ot \xi_{M_2}^{-\rd_2}\ot \dot\rho\dot\nu}
\cong \CL(\CL(\dg,{\rm{id}}),{\rm{id}})^{\xi_{M_1}^{-\rd_1}\ot \tau', 1^{-\rd_1}\ot (\xi_{M_2^2}^{-\rd_2}\ot \dot\nu)}\\
&\cong \CL(\CL_{M_2^2}(\dg,\dot\nu),{\rm{id}})^{\xi_{M_1}^{-\rd_1}\ot \tau''}
\cong \CL_{M_1}(\bar\fg,\bar\tau).
\end{split}\end{equation}

\begin{lemt}\label{lem:iso-last}
One has that
\begin{align*}\uce(\CL(\CL(\dg,{\rm{id}}),{\rm{id}}))^{\uce(\xi_{M_1}^{-\rd_1}\ot \tau'), \uce(1^{-\rd_1}\ot (\xi_{M_2^2}^{-\rd_2}\ot \dot\nu))}
  \cong\uce\(\CL(\CL_{M_2^2}(\dg,\dot\nu),{\rm{id}})\)^{ \uce(\xi_{M_1}^{-\rd_1}\ot\tau'')}.
\end{align*}
\end{lemt}
\begin{proof}
Due to the isomorphisms
\begin{align*}
\uce\(\CL(\dg,{\rm{id}},{\rm{id}})\)^{\uce(1^{-\rd_1}\ot \xi_{M^2_2}^{-\rd_2}\ot \dot\nu)}
\cong \uce(\CL_{1,M_2^2}(\dg,{\rm{id}},\dot\nu))\cong \uce(\CL(\CL_{M_2^2}(\dg,\dot\nu),{\rm{id}})),
\end{align*}
it  suffices to show that the restriction of $\uce(\xi_{M_1}^{-\rd_1}\ot \tau')$ on $\uce(\CL(\CL_{M_2^2}(\dg,\dot\nu),{\rm{id}}))$
coincides with $\uce(\xi_{M_1}^{-\rd_1}\ot \tau'')$.
Set $\fk=\CL(\dg,{\rm{id}},{\rm{id}})$ and $\fk_0=\CL_{1,M_2^2}(\dg,{\rm{id}},\dot\nu)=\CL(\CL_{M_2^2}(\dg,\dot\nu),{\rm{id}})$.
Then by definition one has that
\begin{align*}
&\mathfrak{u}_\fk\circ \uce(\xi_{M_1}^{-\rd_1}\ot \tau')=(\xi_{M_1}^{-\rd_1}\ot \tau')\circ \mathfrak{u}_\fk,\quad
\mathfrak{u}_{\fk_0}\circ \uce(\xi_{M_1}^{-\rd_1}\ot \tau'')=
(\xi_{M_1}^{-\rd_1}\ot \tau'')\circ \mathfrak{u}_{\fk_0},\\
&\mathfrak{u}_{\fk_0}=\mathfrak{u}_\fk\mid_{\uce(\fk_0)=\wh{\CL}_{1,M_2^2}(\dg,{\rm{id}},\dot\nu)}
\quad\te{and}\quad
\xi_{M_1}^{-\rd_1}\ot \tau''=\xi_{M_1}^{-\rd_1}\ot \tau'\mid_{\fk_0}.
\end{align*}
This implies that the restriction of $\uce(\xi_{M_1}^{-\rd_1}\ot \tau')$ on $\uce(\fk_0)$
 covers $\xi_{M_1}^{-\rd_1}\ot \tau''$.
Combining with Proposition \ref{uce} (b), we complete the proof.
 \end{proof}

Now, by  using Lemma \ref{lem:iso-last}, Lemma \ref{conj} (b) and  Lemma \ref{lem:period-uce2},
 we can extend the  isomorphisms given in \eqref{eq:rhonutau}
 to their universal central extensions as follows:
\begin{equation}\begin{split}\label{eq:temp-tss1}
&\quad \uce\(\CL(\dg,{\rm{id}},{\rm{id}})\)^{
    \uce(\xi_{M_1}^{-\rd_1}\ot 1^{-\rd_2}\ot \dot\tau),\uce(1^{-\rd_1}\ot \xi_{M_2}^{-\rd_2}\ot \dot\rho\dot\nu)
}\\
& \cong\uce(\CL(\CL(\dg,{\rm{id}}),{\rm{id}}))^{\uce(\xi_{M_1}^{-\rd_1}\ot \tau'), \uce(1^{-\rd_1}\ot (\xi_{M_2^2}^{-\rd_2}\ot \dot\nu))}\\
&
  \cong\uce\(\CL(\CL_{M_2^2}(\dg,\dot\nu),{\rm{id}})\)^{ \uce(\xi_{M_1}^{-\rd_1}\ot\tau'')}\\
  &\cong\uce\(\CL\(\bar\fg,{\rm{id}}\)\)^{\uce(\xi_{M_1}^{-\rd_1}\ot \bar\tau)}.
  \end{split}
\end{equation}

Combining Lemma \ref{iso1} with \eqref{eq:rhonutau}, we get the isomorphism
\[\CL_{M_1}(\bar\fg,\bar\tau)\cong\CL_M(\bar\fg,\bar\eta).\]
By using
\cite[Theorem 10.1.1 and Corollary 10.1.5]{ABP}, we get that $\bar\tau$ is of the first kind.
Moreover, it follows from \cite[Theorem 13.2.3]{ABP} that the diagram
automorphism $\bar{p}(\bar\tau)$ is conjugate to $\bar{p}(\bar\eta)$.
Thus, one can conclude from Lemma \ref{lem:period-uce} and Lemma \ref{prop:autdec}  that
\begin{align*}
\uce&(\CL(\bar\fg,{\rm{id}}))^{\uce(\xi_{M_1}^{-\rd_1}\ot \bar\tau)}
\cong \uce(\CL(\bar\fg,{\rm{id}}))^{\uce(\xi_{M_1}^{-\rd_1}\ot \bar p(\bar\tau))}\\
&\cong \uce(\CL(\bar\fg,{\rm{id}}))^{\uce(\xi_{M_1}^{-\rd_1}\ot \bar p(\bar\eta))}
\cong \uce(\CL(\bar\fg,{\rm{id}}))^{\uce(\xi_{M}^{-\rd_1}\ot \bar\eta)}.
\end{align*}
Combining with
\eqref{eq:temp-tss1}, we get that
\[ \uce(\CL(\bar\fg,{\rm{id}}))^{\uce(\xi_{M}^{-\rd_1}\ot \bar\eta)}\cong
\uce\(\CL(\dg,{\rm{id}},{\rm{id}})\)^{
    \uce(\xi_{M_1}^{-\rd_1}\ot 1^{-\rd_2}\ot \dot\tau),\uce(1^{-\rd_1}\ot \xi_{M_2}^{-\rd_2}\ot \dot\rho\dot\nu)
}\]
is central closed.
This complets the proof of Theorem \ref{central}.

\section{Proof of Theorem \ref{main2}}\label{sec:proof-MRY}
Throughout this section,
we assume that the diagram automorphism $\mu$ is non-transitive.

\subsection{Root system of $\wh\fg[\mu]$}
In this subsection, we determine the non-isotropic roots in $\wh\fg[\mu]$.
As indicated in \cite[Section 14]{ABP}, this affords an explicit realization of
all nullity 
$2$ reduced extended affine root systems given by Saito (\cite{Saito-EALA}).

Recall that $V=\R\ot_\Z Q$, and we extend $\mu$ to a linear automorphism on $V$ by $\R$-linearity.
We denote by $V_\mu$ the fixed point subspace of $V$ under the isometry $\mu$,
$\pi_\mu:V\rightarrow V_\mu$ the canonical projection of $V$ onto $V_\mu$, and
$\wh{Q}_\mu$ the abelian group $\pi_\mu(Q)\times \Z$.

Define a $Q\times \Z$-grading on $\wh\fg=\oplus_{(\al,n)\in Q\times \Z}\wh\fg_{\al,n}$ by letting
\[t_1^{n_1}\ot x\in  \wh\fg_{\al,n_1},\quad \rk_1\in \wh\fg_{0,0},\quad t_1^{n_1}t_2^{n_2}\rk_1\in \wh\fg_{n_2\delta_2,n_1},
\]
where $x\in \fg_\al$, $\al\in \Delta$, $n_1\in \Z$ and $n_2\in r\Z^\times$.
The above grading induces a $\wh{Q}_\mu$-grading $\wh\fg[\mu]=\oplus_{(\al,n)\in \wh{Q}_\mu}
\wh\fg[\mu]_{\al,n}$ on $\wh\fg[\mu]$ such that for any $(\al,n)\in \wh{Q}_\mu$,
\[\wh\fg[\mu]_{\al,n}=\{x\in \wh\fg[\mu]\cap \wh\fg_{\beta,n}\mid \beta\in Q, \pi_\mu(\beta)=\al\}.
\]
Notice that this is the unique $\wh{Q}_\mu$-grading on $\wh\fg[\mu]$ such that
\begin{align}\label{fgmugrading}
t_1^n\ot e_{i(n)}^\pm\in \wh\fg[\mu]_{\pm\check\al_i,n},\quad t_1^n\ot \al_{i(n)}^\vee\in \wh\fg[\mu]_{0,n},\quad \rk_1\in
\wh\fg[\mu]_{0,0},
\end{align}
for $i\in I$ and $n\in \Z$.

Consider now the following subsets of $\wh{Q}_\mu$:
\begin{align*}&\Phi_\mu=\{(\al,n)\in \wh Q_\mu \mid \wh\fg[\mu]_{\al,n}\ne 0\},\\
&\wh{Q}_\mu^\times=\{(\al,n)\in \wh{Q}_\mu\mid (\al,\al)\ne 0\},\\
&\Phi_\mu^\times=\Phi_\mu\cap\, \wh Q_\mu^\times=
\set{(\al,n)\in\Phi_\mu}{\(\al,\al\)\ne 0}.
\end{align*}
It obvious that $\Phi_\mu\subset \pi_\mu(\Delta)\times \Z$ and so we have
\begin{align}\label{isorootsys1}
\Phi_\mu^\times\subset \pi_\mu(\Delta)^\times\times \Z,
\end{align}
where $\pi_\mu(\Delta)^\times=\{\al\in \pi_\mu(\Delta)|\(\al,\al\)\ne 0\}$.
By definition, for each $i\in I$ we have that $\check{\al}_i=\pi_\mu(\al_i)$.
In addition, for $i\in I$ with $s_i=2$, we have that
 $2\check{\al}_i=\pi_\mu(\al_i+\al_{\mu(i)})$. This shows that
\begin{align}\label{eq:checkpre}
k_i\check{\al}_i\in \pi_\mu(\Delta)^\times,\quad 1\le k_i\le s_i,\ i\in \check{I}.
\end{align}

For $i\in I$, we let $N_i$ be the cardinality of the orbit $\mathcal O(i)$ in $I$ and set $d_i=\frac{N}{N_i}$.
Denote by $\check{W}$  the Weyl group
of the folded GCM $\check{A}$. Then we have the following description of the set $\Phi_\mu^\times$.
\begin{prpt}\label{prop:rootsys1} One has that
\begin{align}
\Phi_\mu^\times=\{(\check{w}(k_i\check{\al}_i),p)\mid \check{w}\in \check{W}, i\in \check{I}, 1\le k_i\le s_i, p\in (k_i-1)d_i+k_id_i\Z\},
\end{align}
and that
\begin{align}
\dim \wh\fg[\mu]_{\al,p}=1,\quad \forall\ (\al,p)\in \Phi_\mu^\times.
\end{align}
\end{prpt}

Before proving Proposition \ref{prop:rootsys1},  we first give a characterization of the set $\pi_\mu(\Delta)^\times$.
This result  is a slight generalization of \cite[Proposition 12.1.16]{ABP}.
\begin{lemt}\label{lem:rootsys0} One has that
\begin{eqnarray}
\pi_\mu(\Delta)^\times = \{\check{w}(k_i\check{\al}_i)\mid \check{w}\in \check{W}, i\in \check{I}, 1\le k_i\le s_i\}.
\end{eqnarray}
\end{lemt}
\begin{proof} For convenience, we set
\[\check{\Delta}^{\mathrm{en}}=\{\check{w}(k_i\check{\al}_i)\mid \check{w}\in \check{W}, i\in \check{I}, 1\le k_i\le s_i\}.\]
We first show that
\begin{eqnarray}\label{eq:check0}
\check{W}(\pi_\mu(\Delta))\subset \pi_\mu(\Delta).
\end{eqnarray}
Let $r_{\check{\al}_i}$, $i\in \check{I}$ denote the reflections associated to $\check{\al}_i$. Note that
 the Weyl group $\check{W}$ is generated by these reflections.
Thus we only need to show that
\begin{eqnarray}\label{eq:check1}
r_{\check{\al}_i}(\pi_\mu(\Delta))\subset\pi_\mu(\Delta),\ i\in\check{I}.
\end{eqnarray}
If $s_i=1$, it is shown in the proof of \cite[Proposition 12.1.16]{ABP} that for each $\alpha\in \Delta$, the following relation holds true
\begin{eqnarray*}
&&r_{\check{\al}_i}(\pi_\mu(\al))=\pi_\mu\(\(\prod_{p\in\mathcal{O}(i)}r_{\al_p}\)(\al)\)\in \pi_\mu(\Delta).
\end{eqnarray*}
If $s_i=2$, then $2\check{\al}_i=\alpha_i+\alpha_{\mu(i)}\in\Delta$ as $ a_{i,\mu(i)}=-1$.
Note that $\pi_\mu(\check{\al}_i)=\check{\al}_i$ and hence
$(\check{\al}_i,\pi_\mu(\al))=(\check{\al}_i,\al)$ for all $\alpha\in \Delta$. This implies
that
\begin{eqnarray*}
&&r_{\check{\al}_i}(\pi_\mu(\al))=\pi_\mu(\al)-2\frac{(\check{\al}_i,\pi_\mu(\al))}{
(\check{\al}_i,\check{\al}_i)}\check{\al}_i
=\pi_\mu\left(\al-2\frac{(\check{\al}_i,\al)}{
(\check{\al}_i,\check{\al}_i)}\check{\al}_i\right)\\
&&\quad\quad=\pi_\mu\left(\al-2\frac{(\check{\al}_i,\al)}{
(2\check{\al}_i,2\check{\al}_i)}2\check{\al}_i\right)=\pi_\mu(r_{2\check{\al}_i}(\al))\in\pi_\mu(\Delta).
\end{eqnarray*}
Thus we complete the proof of the assertion \eqref{eq:check1} and hence
the assertion \eqref{eq:check0}.  Now, as the reflections preserve the bilinear form $(\ ,\ )$,
we have that
\begin{eqnarray*}
\check{W}(\pi_\mu(\Delta)^\times)\subset \pi_\mu(\Delta)^\times.
\end{eqnarray*}
This together with \eqref{eq:checkpre} gives that
\begin{eqnarray*}
\check{\Delta}^{\text{en}}
    \subset \check{W}(\pi_\mu(\Delta)^\times)\subset\pi_\mu(\Delta)^\times.
\end{eqnarray*}

For the reverse inclusion, observe first that any non-zero element $\beta\in\pi_\mu(\Delta)$ can be written uniquely in the form
$\beta=\sum_{i\in\check{I}}n_i\check{\al}_i$, where the $n_i$ are either all non-negative integers
 or all non-positive integers.
Set $\text{ht}\,\beta=\sum_{i\in\check{I}}n_i$. Assume that $\beta\in\pi_\mu(\Delta)^\times$.
We then show that $\beta\in\check{\Delta}^{\mathrm{en}}$ by using induction on $\text{ht}\,\beta$.
Without loss of generality, we may assume that  $\text{ht}\,\beta>0$.
Since $(\beta, \beta)>0$, there are some $i\in\check{I}$ such that $(\beta,\check{\al}_i)>0$ and that
$n_i>0$.
If $r_{\check{\al}_i}(\beta)$ is positive, then we are done by the induction hypothesis.
If $r_{\check{\al}_i}(\beta)$ is negative, then $\beta=q\check{\al}_i$ for some positive integer $q$.
This implies that $\beta=\pi_\mu(\al)$  for some
\[\al=\sum_{p\in\mathcal{O}(i)}m_p\al_p\in\Delta\ \text{with}\ \sum_{p\in\mathcal{O}(i)}m_p=q.\]
If $s_i=1$, then $q$ must equal to $1$ as all $\al_p$, $p\in\mathcal{O}(i)$ are pairwise orthogonal.
If $s_i=2$, then $q$ can be 1 or 2, as $|\mathcal O(i)|=2$ and $a_{i\mu(i)}=-1$.
This completes the proof. 
\end{proof}
As a by-product of Lemma \ref{lem:rootsys0}, we have that
\begin{cort}\label{classicalsr} Let $i,j\in I$ with $\check{a}_{ij}\le 0$. Then for every $p\in \Z$, the elements
\[((1-\check{a}_{ij})\check{\al}_i+\check{\al}_j,p),\ \((s_i+1)\check{\al}_i,p\)\]
are contained in $\wh Q_\mu^\times$ but
not contained in $\Phi_\mu^\times$.
\end{cort}
\begin{proof} By Lemma \ref{lem:rootsys0}, it suffices to show that if $\check{a}_{ij}\le 0$, then $(1-\check{a}_{ij})\check{\al}_i+\check{\al}_j$ is non-isotropic.
Otherwise,
\begin{align*}
0=2\frac{\(\check{\al}_i, (1-\check{a}_{ij})\check{\al}_i+\check{\al}_j\)}{\(\check{\al}_i, \check{\al}_i\)}
=2(1-\check{a}_{ij})+\check{a}_{ij}=2-\check{a}_{ij},
\end{align*}
a contradiction.
\end{proof}

Let $\check{\fg}$ be the subalgebra of $\wh\fg[\mu]$ generated by the elements $\al_{i(0)}^\vee, e_{i(0)}^\pm,\ i\in \check{I}$.
Then by applying Corollary \ref{classicalsr} we have that
\begin{cort}\label{checkfg} The Lie algebra $\check{\fg}$  is isomorphic to the derived subalgebra of the Kac-Moody algebra associated to $\check{A}$.
\end{cort}
\begin{proof} It suffices to check that the elements $\al_{i(0)}^\vee, e_{i(0)}^\pm,\ i\in \check{I}$ satisfy the defining relations
of the derived subalgebra of the Kac-Moody algebra associated to $\check{A}$.
Only the Serre relations $(\mathrm{ad}\ e_{i(0)}^\pm)^{1-\check{a}_{ij}}(e_{i(0)}^\pm)=0$ $(i\ne j)$ are non-trivial. But such
relations are immediate from Corollary \ref{classicalsr}.
\end{proof}

Now we are ready to complete the proof of Proposition \ref{prop:rootsys1}.
Using \eqref{isorootsys1} and Lemma \ref{lem:rootsys0}, we know that any element in $\Phi_\mu^\times$ has the form
\begin{align}\label{rerootele}(\check{w}(k_i\check{\al}_i),p),\quad \check{w}\in \check{W},\ 1\le k_i\le s_i,\ i\in \check{I},\ p\in \Z.\end{align}

 Regard $\wh\fg[\mu]$ as a module of the affine Kac-Moody algebra $\check{\fg}$ (Corollary \ref{checkfg})
via the adjoint action. Then it is integrable, and for each $p\in \Z$, the graded subspace $\wh\fg[\mu]_p$ of $\wh\fg[\mu]$
is a $\check{\fg}$-submodule, where
\begin{align*}
\wh\fg[\mu]_p=\bigoplus_{\check{\al}\in \pi_\mu(\Delta)}
\wh\fg[\mu]_{\check{\al},p}.\end{align*}
Using this and the standard $\mathfrak{sl}_2$-theory, we obtain that $(\check{w}(k_i\check{\al}_i),p)\in \Phi_\mu^\times$ if and only if
$(k_i\check{\al}_i,p)\in \Phi_\mu^\times$. Moreover, we have that $\dim \wh\fg[\mu]_{\check{w}(k_i\check{\al}_i),p}=
\dim \wh\fg[\mu]_{k_i\check{\al}_i,p}$.
So we only need to treat the case that $\check{w}=1$.

We first consider the case that $k_i=1$. Note that for each $i\in \check{I}$,
\[\wh\fg[\mu]_{\check{\al}_i,p}=\C t_1^p \ot e^+_{i(p)}=
\C \sum_{s\in\Z_{N_i}}\left(\sum_{k\in \Z_{d_i}}\xi_{d_i}^{-kr}\right)\xi^{-ps}t_1^p\ot e^+_{\mu^s(i)}.\]
This together with the fact
\[\sum_{k\in \Z_{d_i}}\xi_{d_i}^{-pk}\ne 0\Longleftrightarrow p\in d_i\Z\]
gives that $(\check{\al}_i,p)\in \Phi_\mu^\times$ if and only if $p\in d_i\Z$.
Next, for the case  $k_i=2$ (and hence $s_i=2$), we have that
\[\wh\fg[\mu]_{2\check{\al}_i,p}=\C t_1^p\ot [e^+_i,e^+_{\mu(i)}]_{(p)}.\]
This together with the fact
\[\mu\([e^+_i,e^+_{\mu(i)}]\)=[e^+_{\mu(i)},e_i^+]=-[e^+_i,e^+_{\mu(i)}].\]
gives that $(2\check{\al}_i,p)\in \Phi_\mu^\times$ if and only if $p\in d_i+N\Z$.
Therefore, we complete the proof of Proposition \ref{prop:rootsys1}.

\subsection{Proof of Theorem \ref{main2}}
We start with the following lemma.
\begin{lemt}\label{firsthom} The action
\begin{eqnarray*} c\mapsto \rk_1,\ h_{i,m}\mapsto t_1^m\ot \al_{i(m)}^\vee ,\ x_{i,m}^\pm\mapsto t_1^m\ot e^\pm_{i(m)},\ i\in I,\ m\in \Z,
\end{eqnarray*}
 determines (uniquely) a surjective Lie homomorphism from  $\mathcal{M}(\fg,\mu)$  to $\wh\fg[\mu]$.
\end{lemt}

\begin{proof}One needs to check that the generators $\al_{i(m)}^\vee,  e^\pm_{i(m)}, \rk_1$, $i\in I, m\in \Z$ of $\wh\fg[\mu]$
satisfy the defining relations (T0-T6) of $\mathcal M(\fg,\mu)$.
The relations (T0-T4) follow from a direct verification by using formula \eqref{lem:commutator},
and the relations (T5-T6) are immediate from Proposition \ref{prop:rootsys1}.
\end{proof}

Denote by $\phi_\mu:\mathcal{M}(\fg,\mu)\rightarrow \wh\fg$ the Lie homomorphism given in Lemma \ref{firsthom},
and $\bar\phi_\mu=\psi_\mu\circ \phi_\mu: \mathcal M(\fg,\mu)\rightarrow \CL(\bar\fg,\bar\mu)$ the composition of the map $\phi_\mu$
and the universal central extension $\psi_\mu:\wh\fg[\mu]\rightarrow \CL(\bar\fg,\bar\mu)$. By the universal property of $\psi_\mu$, we know that Theorem \ref{main2} is implied by the
following result.
\begin{prpt} \label{cent}The Lie homomorphism $\bar\phi_\mu:\mathcal{M}(\fg,\mu)\rightarrow \CL(\bar\fg,\bar\mu)$ is a central extension.
\end{prpt}
The rest part of this subsection is devoted to a proof of Proposition \ref{cent}.
Notice that there is a (unique)
 $\wh{Q}_\mu$-grading  $\mathcal{M}(\fg,\mu)=\oplus_{(\al,n)\in \wh{Q}_\mu}\mathcal M(\fg,\mu)_{\al,n}$ on  $\mathcal{M}(\fg,\mu)$ such that
\[\deg c=(0,0),\quad \deg h_{i,m}=(0,m)\quad \text{and}\quad \deg x^{\pm}_{i,m}=(\pm \check{\al}_i,m),\quad i\in I, m\in \Z.\]
We also introduce a $\wh{Q}_\mu$-grading structure $\mathcal L(\bar\fg,\bar\mu)=\oplus_{(\al,n)\in \wh{Q}_\mu}
    \mathcal L(\bar\fg,\bar\mu)_{\al,n}$ so that the quotient map $\psi_\mu:\wh\fg[\mu]\rightarrow \mathcal L(\bar\fg,\bar\mu)$ is graded.
It is obvious that the homomorphism $\phi_\mu$  is $\wh{Q}_\mu$-graded (see \eqref{fgmugrading}) and so is the homomorphism $\bar\phi_\mu$.

Let $\mathcal{M}(\fg,\mu)^\pm$ be the subalgebra of $\mathcal{M}(\fg,\mu)$ generated by $\{x_{i,m}^\pm\,\mid\,i\in I,m\in \Z\}$, and  $\mathcal{M}(\fg,\mu)^0$
the subalgebra of $\mathcal{M}(\fg,\mu)$ generated by $\{h_{i,m}\,\mid\,i\in I,m\in \Z\}$. Then we have the following triangular decomposition
of $\mathcal{M}(\fg,\mu)$, whose proof is straightforward and omitted.
\begin{lemt}\label{mtri}
One has that $\mathcal{M}(\fg,\mu)=\mathcal{M}(\fg,\mu)^+\oplus \mathcal{M}(\fg,\mu)^0\oplus \mathcal{M}(\fg,\mu)^-$.
\end{lemt}

 Recall from Lemma \ref{lem:rootsys0} that $\pi_\mu(\Delta)^\times=\{\check{w}(k_i\check{\al}_i)\mid
 \check{w}\in \check{W}, i\in \check{I}, 1\le k_i\le s_i\}$.

\begin{lemt}\label{supp}Let $(\al,p)\in \wh{Q}_\mu^\times$.
Then the following results hold true

(1)  if $\mathcal{M}(\fg,\mu)_{\al,p}\neq 0$,
then $\al\in\pi_\mu(\Delta)^\times$;

(2) if $\al=\check{w}(\check{\al}_i)$
for some $i\in \check{I}$ and $\check{w}\in \check{W}$, then
the dimension of the graded subspace $\mathcal{M}(\fg,\mu)_{\al,p}$ is $1$ if $p\in d_i\Z$, and
is $0$ otherwise;

(3) if $\al= \check{w}(2\check{\al}_i)$ for some $i\in \check{I}$ with $s_i=2$ and $\check{w}\in \check{W}$,
then the dimension of the graded subspace $\mathcal{M}(\fg,\mu)_{\al,p}$ is $1$  if $p\in d_i+N\Z$, and is $0$ otherwise.
\end{lemt}
\begin{proof} Denote by $\mathcal{M}_0(\fg,\mu)$ the subalgebra of $\mathcal{M}(\fg,\mu)$ generated by the  elements
\[h_{i,0},\ x_{i,0}^\pm,\quad i\in \check{I}.\]
Then one concludes from the relations (T2)-(T5) that $\mathcal{M}_0(\fg,\mu)$ is the derived subalgebra of the Kac-Moody
algebra associated to $\check{A}$. Viewing $\mathcal{M}(\fg,\mu)$ as an $\mathcal{M}_0(\fg,\mu)$-module by the adjoint action,
we see from (T3)-(T6) that the  $\mathcal{M}_0(\fg,\mu)$-module $\mathcal{M}(\fg,\mu)$ is integrable.  Moreover, for each $p\in \Z$, the  subspace
\[\mathcal{M}(\fg,\mu)_p=\bigoplus_{(\alpha,p)\in \wh Q_\mu} \mathcal{M}(\fg,\mu)_{\al,p}\]
of $\mathcal M(\fg,\mu)$ is an $\mathcal{M}_0(\fg,\mu)$-submodule. A standard $\mathfrak{sl}_2$-theory argument gives that
\[\dim \mathcal{M}(\fg,\mu)_{\al,p}=\dim\mathcal{M}(\fg,\mu)_{\check{w}(\al),p},\quad \check{w}\in \check{W}.\]

Assume now that $\mathcal{M}(\fg,\mu)_{\al,p}\neq 0$ for some $(\al,p)\in \wh{Q}_\mu$.
 We now prove that $\alpha\in \pi_\mu(\Delta)^\times$ by using
induction on $\mathrm{ht}\,\al$. Here and as before, $\mathrm{ht}\,\al=\sum_{i\in \check{I}}n_i$ if
$\al=\sum_{i\in\check{I}}n_i\check{\al}_i$. By Lemma \ref{mtri}, the integers $n_i, i\in \check{I}$
are either all non-negative or all non-positive.  We assume  that $\mathrm{ht}\,\al>0$, so that all $n_i$ are non-negative.
Then there exist some $i\in \check{I}$ such that $(\check{\al}_i,\al)>0$ and $n_i>0$.
If $\mathrm{ht}\, r_{\check{\al}_i}(\al)>0$, then we are done by the induction hypothesis.
Otherwise $\mathrm{ht}\, r_{\check{\al}_i}(\al)<0$ and so $\al=k\check{\al}_i$ for some positive integer $k$.
But the relation (T6) forces that $1\le k\le s_i$. This proves the assertion (1).

The assertion (2) is implied by (T0) as $\mathcal{M}(\fg,\mu)_{\check\al_i,p}=\C x_{i,p}^+$. As for the assertion (3), we have that
$N_i=2$ and $\al_{i\mu(i)}=-1$ in this case. Then by the assertion (2) and Lemma \ref{mtri}, we get that
\begin{eqnarray}\label{eq:serre0}
\mathcal{M}(\fg,\mu)_{2\check{\al}_i,p}=
\sum\limits_{\substack{m+n=p\\m,n\in (N/2)\Z}}
    \left[\mathcal{M}(\fg,\mu)_{\check{\al}_i,m},
        \mathcal{M}(\fg,\mu)_{\check{\al}_i,n}\right].
\end{eqnarray}
So the proof of the assertion (3) can be reduced to the proof of the following facts: $\mathcal{M}(\fg,\mu)_{2\check{\al}_i,p}=0$ if
$p\in N\Z$, and $\dim \mathcal{M}(\fg,\mu)_{2\check{\al}_i,p}=1$ if $p\in N/2+N\Z$.
We first show that $\mathcal{M}(\fg,\mu)_{2\check{\al}_i,p}=0$ if
$p\in N\Z$. By \eqref{eq:serre0}, this is implied by
\begin{align}\label{T7}
[x^+_{i,mN/2}, x^+_{i,nN/2}]=0,\quad \te{if}\ m\equiv n\ (\mathrm{mod}\ 2).
\end{align}
Using (T4), we have that
\[[x^+_{i,mN/2}, x^-_{i,nN/2}]=\frac{N}{2}h_{i,(m+n)N/2}+a\,c,\]
for some $a\in \C$.
And by (T3), we have that
\[[h_{i,mN/2}, x^+_{i,nN/2}]=\frac{(2-(-1)^m)N}{2}x^+_{i,(m+n)N/2}.\]
Thus, if $m\equiv n\ (\mathrm{mod}\ 2)$, then
\begin{align*}
[[x^+_{i,0},x^-_{i,0}], [x_{i,mN/2}^+,x^+_{i,nN/2}]]&=\frac{N^2}{2}[x_{i,mN/2}^+,x^+_{i,nN/2}],\\
 [[x^+_{i,mN/2}, x^+_{i,nN/2}], x_{i,0}^-]&=0.
\end{align*}
Combining with (T6), we get that
\begin{align*}&\frac{N^2}{2}[x_{i,mN/2}^+,x^+_{i,nN/2}]=[[x^+_{i,0},x^-_{i,0}], [x_{i,mN/2}^+,x^+_{i,nN/2}]]\\
= &[[x^+_{i,0}, [x_{i,mN/2}^+,x^+_{i,nN/2}]],x^-_{i,0}]+ [x_{i,0}^+,[[x^+_{i,mN/2}, x^+_{i,nN/2}], x_{i,0}^-]]\\
= &[[x^+_{i,0},[x_{i,mN/2}^+,x^+_{i,nN/2}]],x_{i,0}^-]=0,
\end{align*}
This completes the verification of \eqref{T7}.

We now turn to prove the fact that  $\dim \mathcal{M}(\fg,\mu)_{2\check{\al}_i,p}=1$ if $p\in N/2+N\Z$.
It follows from (T3) and (T4) that
\begin{eqnarray}\label{eq:serre1}
&&[ x^-_{i,0},[x^-_{i,0},\mathcal{M}(\fg,\mu)_{2\check \al_i,p}]]\subset \C h_{i,p}.
\end{eqnarray}
It is immediate from the (T2), (T3) and (T4) that $\C x^+_{i,0}+\C x^-_{i,0}+\C h_{i,0}\cong \mathfrak{sl}_2$.
Then by \eqref{eq:serre0}, \eqref{eq:serre1} and the assertion (1),
we find that the space spanned by
\[\mathcal{M}(\fg,\mu)_{k\check{\al}_i,p},\ h_{i,p},\ k=\pm 1, \pm 2\]
is an irreducible $\mathfrak{sl}_2$-module.
This gives that $\dim \mathcal{M}(\fg,\mu)_{2\check{\al}_i,p}\leq 1$.
But one can conclude from Proposition \ref{prop:rootsys1}   that
\[\dim \mathcal{M}(\fg,\mu)_{2\check{\al}_i,p}\ge \dim \wh\fg[\mu]_{2\check{\al}_i,p}=1,\]
as   $\phi_\mu$ is a graded surjective homomorphism. Thus we complete the proof of the assertion (3).
\end{proof}

Now we are in a position to complete the proof of Proposition \ref{cent}.
It follows from Proposition \ref{prop:rootsys1} and
Lemma \ref{supp} that
\begin{align}\label{centpf1}\ker\bar\phi_\mu\subset \mathcal{M}(\fg,\mu)^{\mathrm{iso}}=\bigoplus_{(\al,p)\in \wh{Q}_\mu^0}\mathcal{M}(\fg,\mu)_{\al,p},\end{align}
where \[\wh{Q}_\mu^0=\{(\al,p)\in \wh{Q}_\mu\mid (\al,\al)=0\}.\]
Note that $\wh{Q}^0_\mu+\wh{Q}^\times_\mu\subset \wh{Q}^\times_\mu$, which in particular shows that
\begin{align}\label{centpf2}
[x_{i,m}^\pm, \mathcal{M}(\fg,\mu)^{\mathrm{iso}}]\cap \mathcal{M}(\fg,\mu)^{\mathrm{iso}}=\{0\},\quad
\te{for }i\in I,m\in\Z.
\end{align}
Finally,
 Proposition \ref{cent} is implied  by \eqref{centpf1} and \eqref{centpf2}, as the Lie algebra $\mathcal{M}(\fg,\mu)$ is generated by the elements $x_{i,m}^\pm, i\in I, m\in \Z$.


\begin{thebibliography}{9999999}

\bibitem[AABGP]{AABGP}
B.~Allison, S.~Azam, S.~Berman, Y.~Gao, and A.~Pianzola,
 {\em Extended affine Lie algebras and their root systems},
 {Mem. Amer. Math. Soc.}, {\bf 126}, 1997.

\bibitem[ABFP]{ABFP}
B.~Allison, S.~Berman, J.~Faulkner, and A.~Pianzola,
 Multiloop realization of extended affine {L}ie algebras and Lie
  tori,
 {\em Trans. Amer. Math. Soc.}, {\bf 361} (2009), 4807--4842.

\bibitem[ABGP]{ABGP}
B.~Allison, S.~Berman, Y.~Gao, and A.~Pianzola,
 A characterization of affine Kac-Moody Lie algebras,
 {\em Comm. Math. Phys.}, {\bf 185} (1997), 671--688.

\bibitem[ABP1]{ABP-covering2}
B.~Allison, S.~Berman, and A.~Pianzola,
 {C}overing algebras {II}: {I}somorphism of loop algebras,
 {\em J. Reine Angew. Math.}, {\bf 571} (2004), 39--71.

\bibitem[ABP2]{ABP}
B.~Allison, S.~Berman, and A.~Pianzola,
 Multiloop algebras, iterated loop algebras and extended affine Lie
  algebras of nullity 2,
 {\em J. Eur. Math. Soc.}, {\bf 16} (2014), 327--385.

\bibitem[BGK]{BGK}
S.~Berman, Y.~Gao, and Y.~Krylyuk,
 Quantum tori and the structure of elliptic quasi-simple Lie
  algebras,
 {\em J. Funct. Anal.}, {\bf 135} (1996), 339--389.

\bibitem[BL]{BL}
Y.~Billig and M.~Lau,
 Irreducible modules for extended affine Lie algebras,
 {\em J. Algebra}, {\bf 327} (2011), 208--235.

\bibitem[C]{C}
B.~Cox,
 Two realizations of toroidal ${\mathfrak{sl}_2(\C)}$,
 {\em Contemp. Math.}, {\bf 297} (2002), 47--68.

\bibitem[FJW]{FJW}
I.~Frenkel, N.~Jing, and W.~Wang,
 Quantum vertex representations via finite groups and the Mckay
  correspondence,
 {\em Comm. Math. Phys.}, {\bf 211} (2000), 365--393.

\bibitem[FSS]{FSS}
J.~Fuchs, B.~Schellekens, and C.~Schweigert,
 From Dynkin diagram symmetries to fixed point structures,
 {\em Comm. Math. Phys.}, {\bf 180} (1996), 39--97.

\bibitem[GP]{GP-torsors}
P.~Gille and A.~Pianzola,
 {\em Torsors, reductive group schemes and extended affine Lie algebras},
 {\em Mem. Amer. Math. Soc.}, {\bf 226}, 2013.

\bibitem[GKV]{GKV}
V.~Ginzburg, M.~Kapranov, and E.~Vasserot,
 Langlands reciprocity for algebraic surfaces,
 {\em Math. Res. Lett.}, {\bf 2} (1995), 147--160.

\bibitem[H-KT]{H-KT}
R.~H${\o}$egh-Krohn and B.~Torresani,
 Classification and construction of quasisimple Lie algebras,
 {\em J. Funct. Anal.}, {\bf 89} (1990), 106--136.

\bibitem[J]{J-KM}
N.~Jing,
 Quantum Kac-Moody algebras and vertex representations,
 {\em Lett. Math. Phys.}, {\bf 44} (1998), 261--271.

\bibitem[JM]{JM-fermionic-real-toro}
N.~Jing and K.~Misra,
 {F}ermionic realization of toroidal {L}ie algebras of classical
  types,
 {\em J. Algebra}, {\bf 324} (2010), 183--194.

\bibitem[JMT]{JMT}
N.~Jing, K.~Misra, and S.~Tan,
 Bosonic realizations of higher-level toroidal {L}ie algebras.
 {\em Pracific J. Math.}, {\bf 219} (2005), 285--301.

\bibitem[JMX]{JMX}
N.~ Jing, K.~Misra, and C.~Xu,
 Bosonic realization of toroidal {L}ie algebras of classical types,
 {\em Proc. Amer. Math. Soc.}, {\bf 137} (2009), 3609--3618.

\bibitem[K]{Kac-book}
V.~Kac,
 {\em {Infinite dimensional {L}ie algebras}},
 Cambridge University Press, 1994.

\bibitem[KW]{KW}
V.~Kac and S.~Wang,
 On automorphisms of {K}ac-{M}oody algebras and groups,
 {\em Adv. Math.}, {\bf 92} (1992), 129--195.

\bibitem[MRY]{MRY}
R.~Moody, S.~E.~Rao, and T.~Yokonuma,
 {Toroidal {L}ie algebras and vertex representations},
 {\em Geom. Dedicata}, {\bf 35} (1990), 283--307.

\bibitem[N1]{N2}
E.~Neher,
 {Extended affine {L}ie algebras},
 {\em C.R. Math. Acad. Sci. Soc. R. Can.}, {\bf 26} (2004), 90--96.

\bibitem[N2]{N-uce-functor}
E.~Neher,
 An introduction to universal central extensions of {L}ie
  superalgebras,
 {\em Math. Appl.}, {\bf 555} (2003), 141--166.

\bibitem[Sa]{Saito-EALA}
K.~Saito,
 {E}xtended affine root systems. {I}. {C}oxeter transformations,
 {\em Publ. RIMS., Kyoto Univ.}, {\bf 21} (2009), 75--179.

\bibitem[Sl]{Slo}
P.~Slodowy,
 {Beyond {K}ac-{M}oody algebras, and inside},
 {\em {L}ie algebras and related topics}, {\bf 5} (1986), 361--371.

\bibitem[Su]{Sun-uce}
J.~Sun,
 {{U}niversal central extensions of twisted forms of split simple {L}ie
  algebras over rings},
 {\em J. Algebra}, {\bf 322} (2009), 1819--1829.

\bibitem[T1]{Tan1}
S.~Tan,
 {Principal construction of the toroidal {L}ie algebra of type $A_1$},
 {\em Math. Z.}, {\bf 230} (1999), 621--657.

\bibitem[T2]{Tan2}
S.~Tan,
 {Vertex operator representations for toroidal {L}ie algebra of type
  $B_l$},
 {\em Commun. Algebra}, {\bf 27} (1999), 3593--3618.

\bibitem[VV]{VV-double-loop}
M.~Varagnolo and E.~Vasserot,
 {Double-loop algebras and the {F}ock space},
 {\em Invent. Math.}, {\bf 133} (1998), 133--159.

\end{thebibliography}

\end{document}